\documentclass[12pt, reqno]{amsart}

\usepackage{fullpage}
\usepackage{amssymb}    
\usepackage{latexsym}   
\usepackage{url}        
\usepackage{verbatim}   
\usepackage{array}      
\usepackage{enumitem}   

\theoremstyle{plain}
\newtheorem{theorem}{Theorem}[section]
\newtheorem{lemma}[theorem]{Lemma}
\newtheorem{proposition}[theorem]{Proposition}
\newtheorem{corollary}[theorem]{Corollary}

\newtheorem{problem}[theorem]{Problem}

\theoremstyle{remark}
\newtheorem{example}[theorem]{Example}

\newtheorem*{acknowledgment}{Acknowledgment}

\numberwithin{equation}{section}

\newcommand{\seclabel}[1]{\label{sec:#1}}   
\newcommand{\thmlabel}[1]{\label{thm:#1}}   
\newcommand{\lemlabel}[1]{\label{lem:#1}}   
\newcommand{\corlabel}[1]{\label{cor:#1}}   
\newcommand{\prplabel}[1]{\label{prp:#1}}   
\newcommand{\exmlabel}[1]{\label{exm:#1}}   
\newcommand{\prblabel}[1]{\label{prb:#1}}   
\newcommand{\eqnlabel}[1]{\label{eqn:#1}}   

\newcommand{\secref}[1]{\ref{sec:#1}}   
\newcommand{\thmref}[1]{\ref{thm:#1}}   
\newcommand{\lemref}[1]{\ref{lem:#1}}   
\newcommand{\corref}[1]{\ref{cor:#1}}   
\newcommand{\prpref}[1]{\ref{prp:#1}}   
\newcommand{\prbref}[1]{\ref{prb:#1}}   
\newcommand{\eqnref}[1]{\eqref{eqn:#1}} 

\newcommand{\Aut}{\mathrm{Aut}}		
\newcommand{\Sym}{\mathrm{Sym}}     
\newcommand{\Mlt}{\mathrm{Mlt}}     
\newcommand{\LMlt}{\Mlt_{\lambda}}     
\newcommand{\Inn}{\mathrm{Inn}}		
\newcommand{\Soc}{\mathrm{Soc}}     
\newcommand{\Dih}{\mathrm{Dih}}     

\newcommand{\setof}[2]{\{#1\,|\,#2\}}   

\newcommand{\byeqn}[1]{\overset{\eqnref{#1}}=}  

\newcommand{\ldiv}{\backslash}                    
\newcommand{\rdiv}{\slash}                        
\newcommand{\inv}{^{-1}}                        
\newcommand{\sbl}[1]{\langle#1\rangle}      
\newcommand{\normal}{\mathop{\unlhd}}       
\newcommand{\backnormal}{\mathop{\unrhd}}   
\newcommand{\propnormal}{\lhd}              
\newcommand{\Char}{\mathop{\mathrm{char}}}  

\newcommand{\ad}{\mathrm{ad}}
\newcommand{\id}{\mathrm{id}}
\newcommand{\dia}{\diamond}

\title{The Structure of Automorphic Loops}

\author[M. K. Kinyon]{Michael K. Kinyon}
\address{Department of Mathematics \\
University of Denver \\
Denver, CO 80208 USA}
\email{\url{mkinyon@du.edu}}
\author[K. Kunen]{Kenneth Kunen}
\address{Department of Mathematics \\
University of Wisconsin \\
Madison, WI 57306 USA}
\email{\url{kunen@math.wisc.edu}}
\author[J. D. Phillips]{J. D. Phillips}
\address{Department of Mathematics and Computer Science \\
Northern Michigan University \\
Marquette, Michigan 49855 USA}
\email{\url{jophilli@nmu.edu}}
\author[P. Vojt\v{e}chovsk\'{y}]{Petr Vojt\v{e}chovsk\'{y}$^*$}
\address{Department of Mathematics \\
University of Denver \\
Denver, CO 80208 USA}
\email{\url{petr@math.du.edu}}

\thanks{${}^*$Partially supported by Simons Foundation Collaboration Grant 210176.}

\date{\today}


\begin{document}

\begin{abstract}
Automorphic loops are loops in which all inner mappings are automorphisms. This variety of loops includes, for instance, groups and commutative Moufang loops.

We study uniquely $2$-divisible automorphic loops, particularly automorphic loops of odd order, from the point of view of the associated Bruck loops (motivated by Glauberman's work on uniquely $2$-divisible Moufang loops) and the associated Lie rings (motivated by a construction of Wright). We prove that every automorphic loop $Q$ of odd order is solvable, contains an element of order $p$ for every prime $p$ dividing $|Q|$, and $|S|$ divides $|Q|$ for every subloop $S$ of $Q$.

There are no finite simple nonassociative commutative automorphic loops, and there are no finite simple nonassociative automorphic loops of order less than $2500$. We show that if $Q$ is a finite simple nonassociative automorphic loop then the socle of the multiplication group of $Q$ is not regular. The existence of a finite simple nonassociative automorphic loop remains open.

Let $p$ be an odd prime. Automorphic loops of order $p$ or $p^2$ are groups, but there exist nonassociative automorphic loops of order $p^3$, some with trivial nucleus (center) and of exponent $p$. We construct nonassociative ``dihedral'' automorphic loops of order $2n$ for every $n>2$, and show that there are precisely $p-2$ nonassociative automorphic loops of order $2p$, all of them dihedral.
\end{abstract}

\maketitle

\section{Introduction}
\seclabel{intro}

A \emph{loop} $(Q,\cdot)$ is a set $Q$ with a binary operation $\cdot : Q\times Q\to Q$ such that (i) $(Q,\cdot)$ is a \emph{quasigroup}, that is, for each $a$, $b\in Q$, the equations $ax=b$ and $ya=b$ have unique solutions $x$, $y\in Q$, and (ii) there exists a neutral element $1\in Q$ such that $1x = x1 = x$ for all $x\in Q$. Equivalently, a loop can be viewed as having three binary operations $\cdot, \ldiv, \rdiv$ satisfying the identities $x\ldiv (xy) = y$, $x(x\ldiv y) = y$, $(xy)\rdiv y = x$, $(x\rdiv y)y = x$, $x\rdiv x = y\ldiv y$. Basic references for loop theory are \cite{Bruck, Pflugfelder}.

For $a\in Q$, the \emph{right translation} and \emph{left translation} by $a$ are the bijections $R_a :Q\to Q; x\mapsto xa$ and $L_a :Q\to Q;x\mapsto ax$. These generate the \emph{multiplication group} $\Mlt(Q) = \sbl{R_x, L_x\ |\ x\in Q}$. The \emph{inner mapping group} is the subgroup stabilizing the neutral element, $\Inn(Q) = (\Mlt(Q))_1$.

A loop is \emph{automorphic} (or an \emph{A-loop}) if every inner mapping is an automorphism, that is, if $\Inn(Q) \leq \Aut(Q)$. The study of automorphic loops began in 1956 with Bruck and Paige \cite{BP}. They were particularly interested in \emph{diassociative} automorphic loops, that is, loops in which each $2$-generated subloop is a group. They noted that such loops share many properties with Moufang loops. Shortly thereafter, Osborn showed that commutative diassociative automorphic loops are Moufang \cite{Osborn}. More results showing the Moufang nature of diassociative, automorphic loops were found in \cite{Drapal} and \cite[Thm. 5]{Wright1}. The general case was finally settled in \cite{KKP}: Every diassociative automorphic loop is a Moufang loop.

In recent years, a detailed structure theory has emerged for \emph{commutative} automorphic loops. For instance, the Odd Order, Lagrange and Cauchy Theorems hold for commutative automorphic loops, a finite commutative automorphic loop has order $p^k$ if and only if each element has order a power of $p$, and a finite commutative automorphic loop decomposes as a direct product of a loop of odd order and a loop of order a power of $2$ \cite{JKV1}; there are no finite simple nonassociative commutative automorphic loops \cite{GKN}; for an odd prime $p$, if $Q$ is a finite commutative automorphic $p$-loop then $\Mlt(Q)$ is a $p$-group and $Q$ is centrally nilpotent \cite{Csorgo1,JKV3}; for an odd prime $p$, commutative automorphic loops of order $p$, $p^2$, $2p$, $2p^2$, $4p$, $4p^2$ are groups \cite{JKV2}.

In this paper we lay foundations for the study of automorphic loops. Our understanding is not yet as complete as in the commutative case, but we obtain several significant results, as described below. For notation and terminology, see Section \secref{preliminaries}.

\subsection{Summary of results}

\S \secref{preliminaries} introduces the notation, definitions, and preliminary results concerned mostly with identities valid in automorphic loops.

\S\S \secref{cores}, \secref{bruck}: Motivated by work of Glauberman, we first study certain derived operations on automorphic loops. In \cite{Gl1,Gl2} Glauberman showed that Bruck loops of odd order are solvable and satisfy the Cauchy, Lagrange and Sylow Theorems. He also constructed a Bruck loop $(Q,\circ)$ from a uniquely $2$-divisible Moufang loop $(Q,\cdot)$ by setting $x\circ y = (xy^2x)^{1/2}$, and this  allowed him to transfer the above results from Bruck loops of odd order to Moufang loops of odd order.

We show in three steps that the analog of Glauberman's operation for uniquely $2$-divisible automorphic loops is the operation
\begin{displaymath}
    x\circ y = (x^{-1}\ldiv (y^2x))^{1/2},
\end{displaymath}
which coincides with Glauberman's operation on Moufang loops because Moufang loops are diassociative. First, given any automorphic loop $(Q,\cdot)$, we show that the core $(Q,*)$ defined by
\begin{displaymath}
    x*y = x^{-1}\ldiv (y^{-1}x).
\end{displaymath}
is an involutive quandle. Second, using the core, we show that the set $P_Q = \{P_x\ |\ x\in Q\}$ with $P_x = R_xL_{x^{-1}}^{-1}$ is a twisted subgroup of $\Mlt(Q)$, satisfying $P_xP_yP_y = P_{yP_x}$ and $P_x^n = P_{x^n}$. As is well known, on a uniquely $2$-divisible twisted subgroup $(T,\cdot)$ one can define a Bruck loop $(T,\bullet)$ by
\begin{displaymath}
    x\bullet y = (xy^2x)^{1/2}.
\end{displaymath}
Hence, if $Q$ is uniquely $2$-divisible, $(P_Q,\bullet)$ is a twisted subgroup. Third, the operation $\bullet$ can be transferred from $P_Q$ onto $Q$, yielding the associated Bruck loop $(Q,\circ)$.

A finite automorphic loop is uniquely $2$-divisible if and only if it is of odd order. The above discussion therefore applies to automorphic loops of odd order, and then results of Glauberman on Bruck loops lead to the Lagrange and Cauchy (but not Sylow) Theorems for automorphic loops of odd order.

\S\S \secref{correspondence}, \secref{solvability}: The next ingredient is based on Wright's construction of loops from algebras. Specializing it to a Lie ring $(Q,+,[\cdot,\cdot])$, we can define $(Q,\diamond)$ by
\begin{displaymath}
    x\diamond y = x+y-[x,y].
\end{displaymath}
Then $(Q,\diamond)$ is a loop if and only if in $(Q,+,[\cdot,\cdot])$ the mappings
\begin{equation}\label{Eq:Wright}
    y\mapsto y\pm [y,x]\text{ are invertible for every }x\in Q.\tag{$W_1$}
\end{equation}
Moreover, if $(Q,+,[\cdot,\cdot])$ is a Lie ring satisfying \eqref{Eq:Wright}, then a sufficient condition for $(Q,\diamond)$ to be automorphic is that
\begin{equation}\label{Eq:WrightAutomorphic}
    [[Q,x],[Q,x]]=1\text{ for every $x\in Q$.}\tag{$W_2$}
\end{equation}
In the uniquely $2$-divisible case we obtain the following correspondence: If $(Q,+,[\cdot,\cdot])$ is a uniquely $2$-divisible Lie ring satisfying \eqref{Eq:Wright} and \eqref{Eq:WrightAutomorphic}, then $(Q,\diamond)$ is a uniquely $2$-divisible automorphic loop whose associated Bruck loop $(Q,\circ)$ is an abelian group. Conversely, if $(Q,\cdot)$ is a uniquely $2$-divisible automorphic loop whose associated Bruck loop $(Q,\circ)$ is an abelian group, then $(Q,\circ,[\cdot,\cdot])$ defined by
\begin{displaymath}
    [x,y]=x\circ y\circ (xy)^{-1}
\end{displaymath}
(the inverses in $(Q,\cdot)$ and $(Q,\circ)$ coincide) is a Lie ring satisfying \eqref{Eq:Wright} and \eqref{Eq:WrightAutomorphic}. Moreover, the two constructions are inverse to each other, subrings (resp. ideals) of the Lie ring are subloops (resp. normal subloops) of the automorphic loop, and subloops (resp. normal subloops) of the automorphic loop closed under square roots are subrings (resp. ideals) of the Lie ring.

Taking advantage of the associated Lie rings, we prove the Odd Order Theorem for automorphic loops, we show that automorphic loops of order $p^2$ are groups, and we give examples of automorphic loops of order $p^3$ with trivial nucleus.

\S \secref{simplicity}: Next we investigate finite simple automorphic loops. Since a loop $Q$ is simple if and only if $\Mlt(Q)$ is a primitive permutation group on $Q$, we approach the problem from the direction of primitive groups. In \cite{JKNV} we proved computationally, using the library of primitive groups in GAP, that a finite simple automorphic loop of order less than $2500$ is associative. Here we show that if $Q$ is a finite simple nonassociative automorphic loop then the socle of $\Mlt(Q)$ is not regular, hence, by the O'Nan-Scott theorem, $\Mlt(Q)$ is of almost simple type, of diagonal type or of product type. Whether such a loop exists remains open.

We also prove that characteristically simple automorphic loops behave analogously to characteristically simple groups.

\S\S \secref{mid_nuc}, \secref{dihedral}:
We conclude the paper with a short discussion of middle nuclear extensions and, as an application, with constructions of generalized dihedral automorphic loops. Namely, if $(A,+)$ is an abelian group and $\alpha\in\mathrm{Aut}(A)$ then $\mathbb Z_2\times A$ with multiplication $(i,u)(j,v) = (i + j, ((-1)^ju + v)\alpha^{ij})$ is an automorphic loop. In particular, if $A=\mathbb Z_n$ and $c$ is an invertible element of $\mathbb Z_n$, then $\mathbb Z_2\times \mathbb Z_n$ with multiplication $(i,u)(j,v) = (i + j, ((-1)^ju + v)c^{ij})$ is a dihedral automorphic loop. We show that two such loops are isomorphic if and only if the invertible elements coincide, and we calculate the automorphism groups of these loops.

Cs\"org\H{o} showed in \cite{Csorgo2} that if $Q$ is a finite automorphic loop and $x\in Q$ then $|x|$ divides $|Q|$. This allows us to classify all automorphic loops of order $2p$. There are $p$ such loops up to isomorphism; these are precisely the dihedral automorphic loops corresponding to the $p-1$ invertible elements of $\mathbb Z_p$, and the cyclic group $\mathbb Z_{2p}$.

\section{Preliminaries}
\seclabel{preliminaries}

The inner mapping group $\Inn(Q)$ has a standard set of generators \cite{Bruck}:
\[
R_{x,y} = R_x R_y R_{xy}\inv\,,
\qquad
T_x = R_x L_x\inv\,,
\qquad
L_{x,y} = L_x L_y L_{yx}\inv\,.
\]
Thus automorphic loops can be characterized equationally.

\begin{proposition}[\cite{BP}]
\prplabel{basic}
A loop $Q$ is an automorphic loop if and only if, for all $x$, $y$, $u$, $v\in Q$,
\begin{align*}
(uv)R_{x,y} &= uR_{x,y}\cdot vR_{x,y},   \tag{$A_r$} \\
(uv)L_{x,y} &= uL_{x,y}\cdot vL_{x,y},   \tag{$A_l$} \\
(uv)T_x &= uT_x\cdot vT_x.               \tag{$A_m$}
\end{align*}
\end{proposition}

This means that automorphic loops form a variety in the sense of universal algebra. In particular, subloops and factor loops of automorphic loops are automorphic \cite[Thm. 2.2]{BP}.

A loop $Q$ is \emph{power-associative} if for each $x\in Q$, $\sbl{x}$ is a group. In particular, powers of $x$ are unambiguous, and $x^m x^n = x^{m+n}$ for all $m$, $n\in \mathbb{Z}$.

\begin{proposition} [\text{\cite[Thm. 2.4]{BP}}]
\prplabel{power-associative}
Every automorphic loop is power-associative.
\end{proposition}

We will use the power-associativity of automorphic loops without explicitly referring to Proposition \prpref{power-associative}.

\begin{proposition}[\text{\cite[Thm. 2.5]{BP}}]
\prplabel{commutes}
Let $Q$ be an automorphic loop. Then the following hold for all $x\in Q$, $j$, $k$, $m$, $n\in \mathbb{Z}$.
\begin{align}
L_{x^m}^j L_{x^n}^k &= L_{x^n}^k L_{x^m}^j,  \eqnlabel{LLLL} \\
R_{x^m}^j L_{x^n}^k &= L_{x^n}^k R_{x^m}^j,  \eqnlabel{RLLR} \\
R_{x^m}^j R_{x^n}^k &= R_{x^n}^k R_{x^m}^j.  \eqnlabel{RRRR}
\end{align}
\end{proposition}

\begin{corollary}
\corlabel{RLxxRLxx}
For all $x$ in an automorphic loop $Q$,
\begin{align}
L_{x,x\inv} &= L_{x\inv,x}, \eqnlabel{LxxLxx} \\
R_{x,x\inv} &= R_{x\inv,x}. \eqnlabel{RxxRxx}
\end{align}
\end{corollary}

A loop $Q$ is said to have the \emph{antiautomorphic inverse property} (AAIP) if it has two-sided inverses and satisfies the identity
\[
(xy)\inv = y\inv x\inv  \tag{AAIP}
\]
for all $x$, $y\in Q$. It is also useful to characterize the AAIP in terms of translations and the \emph{inversion mapping} $J: Q\to Q; x\mapsto x\inv$ as either of the following:
\begin{align}
R_x^J &= L_{x\inv}, \eqnlabel{aaip1} \\
L_x^J &= R_{x\inv}. \eqnlabel{aaip2}
\end{align}

\begin{proposition}[\text{\cite[Coro. 6.6]{JKNV}}]
\prplabel{AAIP}
Every automorphic loop has the AAIP.
\end{proposition}

\begin{corollary}
\corlabel{AAIP}
If $Q$ is an automorphic loop, then $J$ normalizes $\Mlt(Q)$ in $\Sym(Q)$.
\end{corollary}

\begin{proof}
Since $\Mlt(Q)$ is generated by left translations, this follows from \eqnref{aaip1} and \eqnref{aaip2} in view of Proposition \prpref{AAIP}
\end{proof}

\begin{lemma}
\lemlabel{RxyLxy}
In an automorphic loop $Q$, the following hold for all $x$, $y\in Q$.
\begin{align}
R_{x,y} &= L_{x\inv,y\inv}, \eqnlabel{RxyLxy} \\
T_x\inv &= T_{x\inv}.        \eqnlabel{Tinv}
\end{align}
\end{lemma}
\begin{proof}
We compute
\[
R_{x,y} = R_{x,y}^J = R_x^J R_y^J (R_{xy}\inv)^J
= L_{x\inv} L_{y\inv} L_{(xy)\inv}\inv = L_{x\inv} L_{y\inv} L_{y\inv x\inv}\inv
= L_{x\inv,y\inv}\,,
\]
where we used $R_{x,y}\in \Aut(Q)$ in the first equality, \eqnref{aaip1} in the third, and (AAIP) in the fourth. This establishes \eqnref{RxyLxy}. For \eqnref{Tinv}, we have
\[
T_x T_{x\inv} = R_x L_x\inv R_{x\inv} L_{x\inv}\inv = R_x R_{x\inv} L_x\inv L_{x\inv}\inv
= R_{x,x\inv} L_{x\inv,x}\inv = R_{x,x\inv} R_{x,x\inv}\inv = \id_Q\,,
\]
where we used \eqnref{RLLR} in the second equality, and \eqnref{RxyLxy} in the fourth.
\end{proof}

To check that a particular loop is automorphic, it is not necessary to verify all of the conditions ($A_r$), ($A_{\ell}$) and ($A_m$):

\begin{proposition}[\text{\cite[Thm. 6.7]{JKNV}}]
\prplabel{halfcheck}
Let $Q$ be a loop satisfying $(A_m)$ and $(A_{\ell})$. Then $Q$ is automorphic.
\end{proposition}

The \emph{left}, \emph{right}, and \emph{middle nucleus} of a loop $Q$ are defined, respectively, by
\begin{align*}
N_{\lambda}(Q) &= \{ a\in Q\mid ax\cdot y = a\cdot xy,\quad \forall x,\,y\in Q\}, \\
N_{\rho}(Q) &= \{ a\in Q\mid xy\cdot a = x\cdot ya,\quad \forall x,\,y\in Q\}, \\
N_{\mu}(Q) &= \{ a\in Q\mid xa\cdot y = x\cdot ay,\quad \forall x,\,y\in Q\},
\end{align*}
and the \emph{nucleus} is $N(Q) = N_{\lambda}(Q)\cap N_{\rho}(Q)\cap N_{\mu}(Q)$. Each of these is a subloop.

Recall that a subloop $S\le Q$ is normal in $Q$, $S\unlhd Q$, if $(S)\varphi = S$ for all $\varphi\in\Inn(Q)$.

\begin{proposition}
\prplabel{nuclei}
Let $Q$ be an automorphic loop. Then
\begin{enumerate}
\item[(i)] $N_{\lambda}(Q) = N_{\rho}(Q) \subseteq N_{\mu}(Q)$, and
\item[(ii)] each nucleus is normal in $Q$.
\end{enumerate}
\end{proposition}

\begin{proof}
The equality $N_{\lambda}(Q) = N_{\rho}(Q)$ is an immediate consequence of the AAIP. Suppose $a \in N_{\lambda}(Q)$. Then $a\inv \in N_{\lambda}(Q)$ and $(x)T_a = a\inv x a$. Now for all $x$, $y\in Q$,
\[
(x\cdot ay)T_a = (x)T_a\cdot (ay)T_a = (a\inv x a)\cdot ya
= a\inv (xa\cdot y) a = (xa\cdot y)T_a\,,
\]
where we used ($A_m$) in the first equality, and the equality of the left and right nuclei in the third. Since $T_a$ is a permutation, we have $x\cdot ay = xa\cdot y$ for all $x$, $y\in Q$, that is, $a\in N_{\mu}(Q)$. This establishes (i). Part (ii) is \cite[Thm. 2.2(iii)]{BP}.
\end{proof}

For a subset $S$ of a loop $Q$, we define the \emph{commutant of} $S$ to be the set
\[
C_Q(S) = \{ a\in Q\mid ax = xa\quad \text{for all}\quad x\in S \}\,.
\]
The \emph{commutant} of $Q$ itself, $C_Q(Q)$ is just denoted by $C(Q)$. (In a group, the commutant of a set is the centralizer of the set and the commutant is the center. However, ``center'' has a narrower meaning in loop theory, and so we adapt operator theory terminology to the present setting.)

\begin{proposition}
\prplabel{commutant}
Let $Q$ be an automorphic loop and let $S\subseteq Q$. Then $C_Q(S) \leq Q$. Furthermore, if $S\normal Q$ then $C_Q(S)\normal Q$. In particular, the commutant $C(Q)$ is a normal subloop of $Q$.
\end{proposition}

\begin{proof}
We have $a\in C_Q(S)$ if and only if $(a)T_x = a$ for all $x\in S$. Thus $C_Q(S)$ is characterized as the intersection of the fixed point sets of all $T_x$, $x\in S$. Since $T_x\in\Aut(Q)$, the fixed point set of $T_x$ is a subloop of $Q$, and $C_Q(S)\le Q$ follows.

Now suppose $S\normal Q$. Fix $a\in C_Q(S)$, $x\in S$, $\varphi\in \Inn(Q)$ and set $y = (x)\varphi\inv\in S$. Then
\[
x(a)\varphi = (y)\varphi(a)\varphi = (ya)\varphi = (ay)\varphi = (a)\varphi(y)\varphi = (a)\varphi x\,,
\]
using $\varphi\in \Aut(Q)$ in the first and fourth equalities and $a\in C_Q(S)$ in the third. Since $x\in S$ was arbitrary, $(a)\varphi\in C_Q(S)$. Thus $C_Q(S)\unlhd Q$.
\end{proof}

We conclude the section with several definitions needed throughout the paper.

A subset $S$ of a loop $Q$ is said to be \emph{characteristic} in $Q$, denoted by $S\Char Q$, if for every $\varphi\in \Aut(Q)$, $(S)\varphi = S$.  A loop is \emph{characteristically simple} if it has no nontrivial characteristic subloops. A loop is \emph{simple} if it has no nontrivial normal subloops.

A loop $Q$ is \emph{solvable} if it has a subnormal series $1= Q_0 \leq \cdots \leq Q_n = Q$, $Q_i \normal Q_{i+1}$, such that each factor loop $Q_{i+1}/Q_i$ is an abelian group.

The \emph{derived subloop} $Q'$ of a loop $Q$ is the smallest normal subloop of $Q$ such that $Q/Q'$ is an abelian group. The derived subloop can be characterized as the smallest normal subloop containing each \emph{commutator} $[x,y]$, defined by $xy\cdot [y,x]=yx$, and each \emph{associator} $[x,y,z]$, defined by $xy\cdot z = (x\cdot yz)[x,y,z]$. Since automorphisms evidently map commutators to commutators and associators to associators, it follows that $Q' \Char Q$.

The \emph{higher derived subloops} are defined in the usual way: $Q^{(2)} = Q'' = (Q')'$, $Q^{(3)} = Q'''$, \emph{etc}. Note that a loop $Q$ is solvable if and only if $Q^{(n)} = 1$ for some $n > 0$.

A \emph{Bruck loop} is a loop satisfying the \emph{left Bol identity} $(x(yx))z = x(y(xz))$ and the \emph{automorphic inverse property} $(xy)\inv = x\inv y\inv$.

\section{Cores and twisted subgroups}
\seclabel{cores}

In an automorphic loop $Q$, we introduce a new binary operation $\ast$ as follows:
\[
x\ast y = x\inv \ldiv (y\inv x) = (x\inv \ldiv y\inv) x  \tag{$\ast$}
\]
for all $x$, $y\in Q$. (The second equality follows from \eqnref{RLLR}.) We will refer to the magma $(Q,\ast)$ as the \emph{core} of the loop $Q$, which should not be confused with the core of a subgroup in group theory. A similar notion was introduced by Bruck \cite{Bruck} for Moufang loops (where the operation can be more simply written as $xy\inv x$) and also in our previous papers \cite{JKV1,JKV2} in the commutative case.

As in \cite{Gl2,JKV1,JKV2}, it is useful to introduce the following permutations for each $x$ in an automorphic loop $Q$:
\[
P_x = R_x L_{x\inv}^{\inv} = L_{x^{-1}}^{-1} R_x,  \tag{P}
\]
where the second equality follows by \prpref{commutes}. Thus the left translation maps of the core $(Q,\ast)$ are just the maps $J P_x$, $x\in Q$; a fact we will use heavily.

\begin{proposition}
\prplabel{coreaut}
Let $Q$ be an automorphic loop with core $(Q,\ast)$. Then for all $x$, $y$, $z\in Q$,
\begin{align}
(y\ast z)R_x &= yR_x\ast zR_x, \eqnlabel{Rcoreaut}\\
(y\ast z)L_x &= yL_x\ast zL_x. \eqnlabel{Lcoreaut}
\end{align}
Therefore $\Mlt(Q) \leq \Aut(Q,\ast)$. In particular, $P_x\in\Aut(Q,*)$ for all $x\in Q$.
\end{proposition}

\begin{proof}
We start with \eqnref{RxyLxy}, which we write as $R_{y,x} = L_{y\inv,x\inv}$, that is, $L_{y\inv} L_{x\inv} L_{(yx)\inv}\inv = R_y R_x R_{yx}\inv$. Rearranging this, we have $L_{x\inv} L_{(yx)\inv}\inv R_{yx} = L_{y\inv}\inv R_y R_x$, or
\begin{equation}
\eqnlabel{PLnorm}
L_{x\inv} P_{yx} = P_y R_x\,.
\end{equation}
Applying both sides of \eqnref{PLnorm} to $z\inv$ yields $(yx)\inv \ldiv [(x\inv z\inv)\cdot yx] = [y\inv \ldiv (z\inv y)]x$. Since $x\inv z\inv = (zx)\inv$ by the AAIP, we have \eqnref{Rcoreaut}.

To establish \eqnref{Lcoreaut}, observe first that $((1\rdiv y)x^{-1})^{-1} = x(1\rdiv y)^{-1} = xy$ by AAIP, and so $R_y\inv R_{x\inv} R_{xy}$ is an inner mapping, hence an automorphism. Thus
\[
R_y\inv R_{x\inv} R_{xy} = (R_y\inv R_{x\inv} R_{xy})^J
= (R_y\inv)^J R_{x\inv}^J R_{xy}^J = L_{y\inv}\inv L_x L_{(xy)\inv}\,,
\]
using \eqnref{aaip1} and \eqnref{aaip2}. Rearranging, we have $R_{x\inv} R_{xy} L_{(xy)\inv}\inv = R_y L_{y\inv}\inv L_x$, or
\begin{equation}
\eqnlabel{PRnorm}
R_{x\inv} P_{xy} = P_y L_x\,.
\end{equation}
Applying both sides of \eqnref{PRnorm} to $z\inv$ yields $(xy)\inv \ldiv [(z\inv x\inv)\cdot xy] = x[y\inv \ldiv (z\inv y)]$. Since $z\inv x\inv = (xz)\inv$ by the AAIP, we are finished.
\end{proof}

\begin{lemma}
For all $x$ in an automorphic loop $Q$,
\begin{equation}
\eqnlabel{Pinv}
P_x^J = P_x^{\inv} = P_{x\inv}\,.
\end{equation}
Thus in the core $(Q,\ast)$, the following holds for all $x$, $y\in Q$:
\begin{equation}
\eqnlabel{coreinv}
(x \ast y)\inv = x\inv \ast y\inv\,.
\end{equation}
\end{lemma}
\begin{proof}
We have $P_x^J = R_x^J (L_{x\inv}\inv)^J = L_{x\inv} R_x\inv = P_x\inv$, using \eqnref{aaip1} and \eqnref{aaip2}. Also,
\[
P_x P_{x\inv} = R_x L_{x\inv}\inv R_{x\inv} L_x\inv = R_x L_x\inv R_{x\inv} L_{x\inv}\inv
= T_x T_{x\inv} = \id_Q\,,
\]
using \eqnref{RLLR} and \eqnref{LLLL} in the second equality and \eqnref{Tinv} in the fourth. This establishes \eqnref{Pinv}. Then \eqnref{coreinv} follows, since $(x\ast y)\inv = yJP_xJ = yP_x^J = (y\inv)JP_{x\inv} = x\inv \ast y\inv$.
\end{proof}

\begin{theorem}
\thmlabel{coreprops}
Let $Q$ be an automorphic loop with core $(Q,\ast)$. Then $(Q,\ast)$ is an involutive quandle, that is, the following properties hold:
\begin{enumerate}
\item[(i)] $x\ast x = x$ for all $x\in Q$,
\item[(ii)] $x\ast (x\ast y) = y$ for all $x$, $y\in Q$,
\item[(iii)] $x\ast (y\ast z) = (x\ast y)\ast (x\ast z)$ for all $x$, $y$, $z\in Q$.
\end{enumerate}
\end{theorem}
\begin{proof}
Part (i) is clear from the definition of $\ast$. For (ii), $x\ast (x\ast y) = yJP_xJP_x = yP_x^J P_x = y$ by \eqnref{Pinv}. For (3),
\[
x\ast (y\ast z) = (y\ast z)JP_x = (y\inv \ast z\inv)P_x = (y\inv)P_x \ast (z\inv)P_x
= (x\ast y)\ast (x\ast z)\,,
\]
using \eqnref{coreinv} and Proposition \prpref{coreaut}.
\end{proof}

Recall that a subset $A$ of a group $G$ is said to be a \emph{twisted subgroup} of $G$ if (i) $1\in A$, (ii) $a\in A$ implies $a\inv \in A$, and (iii) $a,b\in A$ implies $aba\in A$.

In an automorphic loop $Q$, let $P_Q = \setof{P_x}{x\in Q}$.

\begin{proposition}
\prplabel{twisted}
Let $Q$ be an automorphic loop. Then $P_Q$ is a twisted subgroup of $\Mlt(Q)$. In particular,
\begin{equation}
\eqnlabel{Ptwisted}
P_x P_y P_x = P_{yP_x}
\end{equation}
for all $x$, $y\in Q$.
\end{proposition}
\begin{proof}
Clearly $\id_Q = P_1\in P_Q$. For $x\in Q$, $P_x\inv \in P_Q$ by \eqnref{Pinv}. Since $JP_x \in \Aut(Q,\ast)$ by Theorem \thmref{coreprops}(iii), we have $zJP_y JP_x = (y\ast z)JP_x = yJP_x\ast zJP_x = zJP_x JP_{yJP_x}$ for all $x$, $y$, $z\in Q$. Thus $P_y^J P_x = P_x^J P_{(y\inv)P_x}$. By \eqnref{Pinv}, we deduce $P_x P_{y\inv} P_x = P_{(y\inv)P_x}$. Replacing $y$ with $y\inv$, we have \eqnref{Ptwisted}.
\end{proof}

\begin{corollary}
\corlabel{Ppowers}
Let $Q$ be an automorphic loop. Then for all $x\in Q$ and $n\in \mathbb{Z}$,
\begin{equation}
\eqnlabel{Ppowers}
P_x^n = P_{x^n}\,.
\end{equation}
\end{corollary}
\begin{proof}
Since $(x^n)P_x = x^{n+2}$, the desired result follows for $n\geq 0$ by an easy induction using \eqnref{Ptwisted}. For $n < 0$, apply \eqnref{Pinv}.
\end{proof}

Although we have no application for the following result, we mention it for the sake of completeness:

\begin{proposition}
\prplabel{Pnormal}
Let $Q$ be an automorphic loop. Then $\langle P_Q\rangle \normal \Mlt(Q)$.
\end{proposition}
\begin{proof}
By \eqnref{PLnorm}, we have for each $x,y\in Q$, $R_x\inv P_y R_x = R_x\inv L_{x\inv} P_{yx} = P_x\inv P_{yx} \in \langle P_Q\rangle$. By \eqnref{PRnorm}, we have for each $x,y\in Q$, $L_x\inv P_y L_x = L_x\inv R_{x\inv} P_{xy} = P_{x\inv} P_{xy}\in \langle P_Q\rangle$. Since $\Mlt(Q)$ is generated by all $R_x$, $L_x$, $x\in Q$, we have the desired result.
\end{proof}

\section{Uniquely $2$-divisible automorphic loops}
\seclabel{bruck}

A loop $Q$ is said to be \emph{uniquely} $2$-\emph{divisible} if the squaring map $x\mapsto x^2$ is a permutation of $Q$.

\begin{lemma}
\lemlabel{Pcorrespond}
Let $Q$ be a uniquely $2$-divisible automorphic loop. Then $Q\to P_Q; x\mapsto P_x$ is a bijection.
\end{lemma}
\begin{proof}
To see that the map is one-to-one, suppose $P_x = P_y$. Applying both sides to $1$, we obtain $x^2 = y^2$. By unique $2$-divisibility, $x = y$.
\end{proof}

It is well known that a uniquely $2$-divisible twisted subgroup $T$ of a group $G$ can be turned into a Bruck loop $(T,\bullet)$ by setting
\begin{displaymath}
    a\bullet b = (ab^2a)^{1/2}\tag{$\bullet$}.
\end{displaymath}
See \cite[Lem. 4.5]{FKP}, for instance.

In a uniquely $2$-divisible automorphic loop $Q$, the set $P_Q$ is a uniquely $2$-divisible twisted subgroup of $\Mlt(Q)$ by Proposition \prpref{twisted} and Corollary \corref{Ppowers}, noticing that $P_x^{1/2} = P_{x^{1/2}}$ for all $x\in Q$. Thus we can define
\[
P_x\bullet P_y = [P_x P_y^2 P_x]^{1/2} = P_{(y^2)P_x}^{1/2} = P_{[(y^2)P_x]^{1/2}},
\]
making $(P_Q,\bullet)$ into a Bruck loop.

Upon defining $(Q,\circ)$ on $Q$ by
\[
x\circ y = [(x\inv \ldiv y^2) x]^{1/2} = [(y^2)P_x]^{1/2}\,,  \tag{$\circ$}
\]
we see that the bijection $(Q,\circ)\to (P_Q,\bullet)$; $x\mapsto P_x$ is an isomorphism of magmas. Thus $(Q,\circ)$ is a Bruck loop, the \emph{Bruck loop associated with the uniquely $2$-divisible automorphic loop $Q$}.

We have established most of the following:

\begin{proposition}
\prplabel{bruckloop}
Let $Q$ be a uniquely $2$-divisible automorphic loop. Then $(Q,\circ)$ defined by $(\circ)$ is a Bruck loop. Powers in $(Q,\circ)$ coincide with powers in $Q$. Any subloop of $Q$ which is closed under square roots is a subloop of $(Q,\circ)$.
\end{proposition}
\begin{proof}
We already showed that $(Q,\circ)$ is a Bruck loop. Powers of $x$ in $(Q,\circ)$ correspond to powers of $P_x$ in $(P_Q,\bullet)$. But these coincide with powers of $P_x$ in $\Mlt(Q)$ \cite[Lem 4.5]{FKP}. By Corollary \corref{Ppowers}, we conclude that powers in $(Q,\circ)$ coincide with powers in $Q$. In Bruck loops, the left and right divisions can be expressed in terms of the multiplication and inversion: $x\ldiv_{\circ} y = x\inv \circ y$ and $x\rdiv_{\circ} y = y\inv\circ ((y\circ x)\circ y\inv)$. Thus the claim about subloops follows directly from ($\circ$).
\end{proof}

Note that $x\circ y = [(x\inv\ldiv y^2)x]^{1/2} = [x\inv\ldiv(y^2x)]^{1/2}$ by Proposition \prpref{commutes}

\begin{proposition}
Let $Q$ be a uniquely $2$-divisible automorphic loop. Then the core $(Q,\ast)$ is a quasigroup.
\end{proposition}
\begin{proof}
This follows immediately from the unique $2$-divisibility, the fact that $(Q,\circ)$ is a loop, and the observation $x\ast y = (x\circ y^{-1/2})^2$.
\end{proof}

The \emph{left multiplication group} $\LMlt(Q)$ of a loop $Q$ is the group $\sbl{L_x\ |\ x\in Q}\le\Mlt(Q)$.

\begin{lemma}
\lemlabel{LMltBruck}
Let $Q$ be a uniquely $2$-divisible automorphic loop with associated Bruck loop $(Q,\circ)$. Then $\LMlt(Q,\circ)$ is conjugate in $\Sym(Q)$ to $\sbl{P_Q}$.
\end{lemma}
\begin{proof}
Let $\sigma : Q\to Q; x\mapsto x^2$ denote the squaring permutation. For each $x\in Q$, the left translation $y\mapsto x\circ y$ is just $\sigma P_x \sigma^{-1}$. This establishes the desired result.
\end{proof}

We will need the following easy observation later.

\begin{lemma}
\lemlabel{bruck-aut}
Let $Q$ be a uniquely $2$-divisible automorphic loop with associated Bruck loop $(Q,\circ)$. Then $\Aut(Q)\leq \Aut(Q,\circ)$. In particular, every inner mapping of $Q$ acts as an automorphism of $(Q,\circ)$.
\end{lemma}

Next, we prove the Lagrange and Cauchy Theorems for automorphic loops of odd order. First, we must show that for finite automorphic loops, the notions of unique $2$-divisibility and having odd order coincide. In fact, this is true more generally for finite power-associative loops.

\begin{lemma}
\lemlabel{odd-2div}
Let $Q$ be a finite loop with two-sided inverses.
\begin{enumerate}
\item[(i)] If $Q$ is uniquely $2$-divisible, then $Q$ has odd order.
\item[(ii)] If $Q$ has odd order and the AAIP, then $Q$ has no
elements of order $2$. If $Q$ is also power-associative,
then $Q$ is uniquely $2$-divisible.
\end{enumerate}
\end{lemma}
\begin{proof}
Suppose $Q$ is uniquely $2$-divisible. Then the inversion mapping $J$ fixes only the identity element. Since $J$ has order $2$, the set of nonidentity elements of $Q$ must have even order, and so $Q$ has odd order. This proves (i).

Now assume $Q$ has odd order and the AAIP, and suppose $c\in Q$ satisfies $c^2 = 1$. By the AAIP, if $xy = c$ then $c = c\inv = (xy)\inv = y\inv x\inv$. Thus the set $K = \setof{(x,y)}{xy = c}$ is invariant under the mapping $\phi: Q^2\to Q^2;(x,y)\mapsto (y\inv,x\inv)$. Since $\phi$ is involutive and $|K|$ is odd, $\phi$ has a fixed point $(x,y)\in K$. This point satisfies $x\inv = y$, so that $1 = x x\inv = c$. This establishes the first part of (ii), and the remaining assertion is clear.
\end{proof}

\begin{corollary}
\corlabel{odd-2div}
A finite automorphic loop is uniquely $2$-divisible if and only if it has odd order.
\end{corollary}

\begin{corollary}
\corlabel{even}
Let $Q$ be a finite automorphic loop of even order. Then $Q$ contains an element of order $2$.
\end{corollary}
\begin{proof}
Otherwise, every element of $Q$ would have odd order, so that $Q$ would be uniquely $2$-divisible, and hence have odd order.
\end{proof}

\begin{lemma}
\lemlabel{subloops}
Let $Q$ be an automorphic loop of odd order with associated Bruck loop $(Q,\circ)$. If $S$ is a subloop of $Q$, then $S$ is a subloop of $(Q,\circ)$.
\end{lemma}
\begin{proof}
In this case, the square root of any element is a positive integer power of that element, and so subloops are closed under taking square roots. Then Proposition \prpref{bruckloop} applies.
\end{proof}

\begin{theorem}[Lagrange Theorem]
\thmlabel{lagrange}
Let $Q$ be an automorphic loop of odd order. If $S\leq Q$, then $|S|$ divides $|Q|$.
\end{theorem}
\begin{proof}
By Lemma \lemref{subloops}, $S$ is a subloop of the associated Bruck loop $(Q,\circ)$. The result follows from \cite[Cor. 4]{Gl1}.
\end{proof}

Note that Theorem \thmref{lagrange}, sometimes called the weak Lagrange property, implies what is known as the strong Lagrange property for automorphic loops of odd order: if $T\leq S\leq Q$, then $|T|$ divides $|S|$. This is because subloops of automorphic loops of odd order are themselves automorphic loops of odd order.

\begin{theorem}[Cauchy Theorem]
\thmlabel{cauchy}
Let $Q$ be an automorphic loop of odd order. If a prime $p$ divides $|Q|$, then $Q$ contains an element of order $p$.
\end{theorem}
\begin{proof}
By \cite[Coro 1, p. 394]{Gl1}, the associated Bruck loop $(Q,\circ)$ contains an element of order $p$ and thus so does $Q$ by Proposition \prpref{bruckloop}.
\end{proof}

\begin{corollary}
\corlabel{orderpisgroup}
Every automorphic loop of prime order is a group.
\end{corollary}
\begin{proof}
This is trivial for $p = 2$, while for $p$ odd, it follows from Theorem \thmref{cauchy}.
\end{proof}

\section{A correspondence with Lie rings}
\seclabel{correspondence}

Following Wright \cite{Wright2}, if $(A,+,\cdot)$ is an algebra (over some field), define $(A,\diamond)$ by $x\diamond y = x + y - xy$. By \cite[Prop. 8]{Wright2}, $(A,\diamond)$ is a loop if and only if the mappings $y\mapsto y-yx$, $y\mapsto y-xy$ are bijections of $A$. We will now specialize this construction to Lie rings, and establish its partial inverse.

Recall that a \emph{Lie ring} $(Q,+,[\cdot,\cdot])$ is an abelian group $(Q,+)$ such that the bracket $[\cdot,\cdot]$ is biadditive, satisfies the Jacobi identity $[x,[y,z]]+[y,[z,x]]+[z,[x,y]]=0$, and is alternating, that is, $[x,x]=0$. Consequently, Lie rings are skew-symmetric, $[x,y]=-[y,x]$.

As usual, for $x\in Q$ define $\mathrm{ad}(x):Q\to Q$; $y\mapsto [y,x]$. Thanks to skew-symmetry, the mappings from Wright's construction take on the form
\begin{align*}
    r_x &= \mathrm{id}_Q-\mathrm{ad}(x);\; y\mapsto y-[y,x],\\
    \ell_x &= \mathrm{id}_Q+\mathrm{ad}(x);\; y\mapsto y+[y,x].
\end{align*}
Note that all $r_x$, $\ell_x$ are homomorphisms of $(Q,+)$.

In this context, Wright's construction can be stated as follows:

\begin{lemma}\label{Lm:LieWright}
Let $(Q,+,[\cdot,\cdot])$ be a Lie ring. Then $(Q,\diamond)$ defined by
\begin{displaymath}
    x\diamond y = x + y - [x,y]\tag{$\diamond$}
\end{displaymath}
is a loop (with neutral element $0$) if and only if $(Q,+,[\cdot,\cdot])$ satisfies \eqref{Eq:Wright}, that is, if and only if the mappings $r_x$, $\ell_x$ are invertible for every $x\in Q$.

For $x\in Q$ let $R_x^\diamond$, $L_x^\diamond$ be the right and left translation by $x$ in the groupoid $(Q,\diamond)$. Then
\begin{align*}
    &y\diamond x = yR_x^\diamond = x+yr_x,\\
    &x\diamond y = yL_x^\diamond = x+y\ell_x.
\end{align*}
If $(Q,\diamond)$ is a loop then also
\begin{align*}
    y(R_x^\diamond)^{-1} &= (y-x)r_x^{-1},\\
    y(L_x^\diamond)^{-1} &= (y-x)\ell_x^{-1},\\
    R_y^\diamond R_z^\diamond (R_{y\diamond z}^\diamond)^{-1} &= r_yr_zr_{y\diamond z}^{-1},\\
    L_y^\diamond L_z^\diamond (L_{z\diamond y}^\diamond)^{-1} &= \ell_y\ell_z\ell_{z\diamond y}^{-1},\\
    R_y^\diamond (L_y^\diamond)^{-1} &= r_y\ell_y^{-1}.
\end{align*}
\end{lemma}
\begin{proof}
The first part of the statement is a special case of Wright's result, and the formulae for the translations $R_x^\diamond$, $L_x^\diamond$ follow from ($\diamond$) and skew-symmetry. For the rest of the proof suppose that $(Q,\diamond)$ is a loop. We immediately get the formulae for $(R_x^\diamond)^{-1}$ and $(L_x^\diamond)^{-1}$. Finally, $xR_y^\diamond R_z^\diamond (R_{y\diamond z}^\diamond)^{-1} = (z+(y+xr_y)r_z - (y\diamond z))r_{y\diamond z}^{-1} = (z+yr_z+xr_yr_z - (z+yr_z))r_{y\diamond z}^{-1} = xr_yr_zr_{y\diamond z}^{-1}$, $xL_y^\diamond L_z^\diamond (L_{z\diamond y}^\diamond)^{-1} = (z+(y+x\ell_y)\ell_z - (z\diamond y))\ell_{z\diamond y}^{-1} = (z+y\ell_z + x\ell_y\ell_z - (z+y\ell_z))\ell_{z\diamond y}^{-1} = x\ell_y\ell_z\ell_{z\diamond y}^{-1}$, and $xR_y^\diamond (L_y^\diamond)^{-1} = ((y+xr_y)-y)\ell_y^{-1} = xr_y\ell_y^{-1}$.
\end{proof}

The Lie ring construction sometimes yields automorphic loops:

\begin{proposition}
\label{Pr:LieToAut}
Let $(Q,+,[\cdot,\cdot])$ be a Lie ring satisfying \eqref{Eq:Wright} and \eqref{Eq:WrightAutomorphic}. Then $(Q,\diamond)$ defined by $(\diamond)$ is an automorphic loop, and the commutant and nuclei of $(Q,\diamond)$ are given by
\begin{align*}
    C(Q,\dia) &= \setof{a\in Q}{2[a,x]=0,\ \forall x\in Q}\\
    N_{\lambda}(Q,\dia) &= \setof{a\in Q}{[[a,x],y] = 0,\ \forall x,\,y\in Q}\\
    N_{\mu}(Q,\dia) &= \setof{a\in Q}{[[x,y],a] = 0,\ \forall x,\,y\in Q}.
\end{align*}
In particular, $(Q,\diamond)$ is a group if and only if $[[x,y],z]=0$ for all $x$, $y$, $z\in Q$.
\end{proposition}
\begin{proof}
By Lemma \ref{Lm:LieWright}, $(Q,\diamond)$ is a loop. For all $x$, $y$, $z\in Q$, we have
\begin{align*}
    x r_z\dia y r_z
    &= x - [x,z] + y - [y,z] - [x-[x,z],y-[y,z]] \\
    &= x + y - [x,y] - [x + y,z] + [x,[y,z]] + [[x,z],y] - [[x,z],[y,z]] \\
    &= x + y - [x,y] - [x + y,z] + [[x,y],z] \\
    &= (x\dia y) r_z \,,
\end{align*}
where we have used both the Jacobi identity and the condition \eqref{Eq:WrightAutomorphic} in the third equality. Thus for each $z\in Q$ we have $r_z\in\Aut(Q,\dia)$. Similarly, $\ell_z\in\Aut(Q,\dia)$. By Lemma \eqref{Lm:LieWright}, the standard generators $R_y^\diamond R_z^\diamond (R_{y\diamond z}^\diamond)^{-1}$, $L_y^\diamond L_z^\diamond (L_{z\diamond y}^\diamond)^{-1}$ and $R_y^\diamond (L_y^\diamond)^{-1}$ of $\Inn(Q,\diamond)$ are elements of $\sbl{r_x,\,\ell_x\ |\ x\in Q}\le \Aut(Q,\dia)$, and hence $(Q,\diamond)$ is an automorphic loop.

The characterization of the commutant is clear from ($\dia$). For the nuclei, we compute
\[
\big((x\dia y)\dia z\big) - \big(x\dia (y\dia z)\big) =
[[x,y],z] - [x,[y,z]] = [[x,z],y]\,,
\]
using the Jacobi identity. Thus a triple $x$, $y$, $z$ associates in $(Q,\dia)$ if and only if $[[x,z],y] = 0$. All remaining claims easily follow.
\end{proof}

\begin{corollary}
\corlabel{glie}
Let $Q$ be a Lie ring satisfying \eqref{Eq:Wright} and \eqref{Eq:WrightAutomorphic}, and let $(Q,\dia)$ be the corresponding automorphic loop.
\begin{enumerate}
\item[(i)] If $Q$ has characteristic $2$, then $(Q,\dia)$ is commutative.
\item[(ii)] If the abelian group $(Q,+)$ is uniquely $2$-divisible, then $C(Q,\diamond) = Z(Q,\diamond)$ is equal to the center of the Lie ring $Q$.
\end{enumerate}
\end{corollary}

For the rest of this section we will be concerned with the question of whether it is possible to invert the construction of Proposition \ref{Pr:LieToAut} to obtain a Lie ring satisfying \eqref{Eq:Wright} and \eqref{Eq:WrightAutomorphic} from an automorphic loop. We identify suitable subclasses of Lie rings and automorphic loops when this is indeed the case.

For the rest of this section we will deal with uniquely $2$-divisible automorphic loops $Q$ for which the associated Bruck loop $(Q,\circ)$ is a group, hence an abelian group. For $x$ in such a loop $Q$, define the inner mapping
\[
\phi_x = R_x P_{x^{1/2}}\inv.   \tag{$\phi$}
\]
We will make heavy use of the fact that $\phi_x\in \Aut(Q)$, often without explicit reference.

\begin{lemma}
\lemlabel{op_corres}
Let $Q$ be a uniquely $2$-divisible automorphic loop for which the associated Bruck loop $(Q,\circ)$ is an abelian group. For all $x$, $y\in Q$, the following identities hold:
\begin{align}
xy &= (x)\phi_y \circ y,  \eqnlabel{op_corres} \\
x\circ (y\inv)\phi_x &= y\inv \circ (x)\phi_y.  \eqnlabel{flip}
\end{align}
\end{lemma}
\begin{proof}
For all $x$, $y\in Q$,
\[
x^2 y^2 = (x^2)R_{y^2} P_y\inv P_y = (x^2)\phi_{y^2} P_y
= [((x)\phi_{y^2})^2]P_y = [(x)\phi_{y^2} \circ y]^2\,.
 \]
Since $(Q,\circ)$ is an abelian group and powers in $(Q,\cdot)$, $(Q,\circ)$ coincide, we have $x^2 y^2 = ((x)\phi_{y^2})^2 \circ y^2 = (x^2)\phi_{y^2}\circ y^2$. Replacing $x$ with $x^{1/2}$ and $y$ with $y^{1/2}$, we obtain \eqnref{op_corres}.

Now using AAIP, we have
\[
(y\inv)\phi_{x\inv}\circ x\inv = y\inv x\inv = (xy)\inv = [(x)\phi_y\circ y]\inv
= (x\inv)\phi_y \circ y\inv\,.
\]
Replacing $x$ with $x\inv$, we obtain \eqnref{flip}.
\end{proof}

\begin{lemma}
\lemlabel{Pgroup}
Let $Q$ be a uniquely $2$-divisible automorphic loop for which the associated Bruck loop $(Q,\circ)$ is an abelian group. Then $P_Q = \sbl{P_Q}$ is an abelian group isomorphic to $(Q,\circ)$. In particular, for all $x$, $y\in Q$,
\begin{equation}
\eqnlabel{Pgroup}
    P_x P_y = P_{x\circ y}\,.
\end{equation}
\end{lemma}
\begin{proof}
Since $(Q,\circ)$ is an abelian group, $(Q,\circ)\cong \LMlt(Q,\circ)\cong \sbl{P_Q}$ by Lemma \lemref{LMltBruck}. For \eqnref{Pgroup}, we have $P_{x\circ y} = P_x\bullet P_y = (P_x P_y^2 P_x)^{1/2} = (P_x^2 P_y^2)^{1/2} = P_x P_y$, since $\sbl{P_Q}$ is an abelian group.
\end{proof}

\begin{lemma}
\lemlabel{phigenerate}
Let $Q$ be a uniquely $2$-divisible automorphic loop for which the associated Bruck loop $(Q,\circ)$ is an abelian group. Then $\sbl{\phi_x\ |\ x\in Q} = \Inn(Q)$.
\end{lemma}
\begin{proof}
One inclusion is obvious. We have
\begin{equation}
\eqnlabel{phigen1}
R_{x,y} = R_x R_y R_{xy}\inv = \phi_x P_{x^{1/2}} \phi_y P_{y^{1/2}} P_{(xy)^{1/2}}\inv \phi_{xy}\inv
= \phi_x \phi_y P_{(x^{1/2})\phi_y} P_{y^{1/2}} P_{(xy)^{1/2}}\inv \phi_{xy}\inv\,,
\end{equation}
since $\phi_y\in \Aut(Q)$. Now by \eqnref{Pgroup}, $P_{(x^{1/2})\phi_y} P_{y^{1/2}} = P_{(x^{1/2})\phi_y\circ y^{1/2}}$. By the fact that $(Q,\circ)$ is an abelian group and \eqnref{op_corres}, $(x^{1/2})\phi_y\circ y^{1/2} = [(x)\phi_y\circ y]^{1/2} = (xy)^{1/2}$. Thus \eqnref{phigen1} reduces to $R_{x,y} = \phi_x \phi_y \phi_{xy}\inv$. By \eqnref{RxyLxy}, $L_{x,y} = R_{x\inv,y\inv} = \phi_{x\inv} \phi_{y\inv} \phi_{x\inv y\inv}\inv$. Finally,
\[
T_x = R_x L_x\inv = \phi_x P_{x^{1/2}} L_x\inv = \phi_x P_{x^{1/2}} P_{x\inv} R_{x\inv}\inv = \phi_x P_{x^{-1/2}} R_{x\inv}\inv = \phi_x \phi_{x\inv}\inv\,,
\]
where we used \eqnref{Pgroup} and $x^{1/2}\circ x\inv = x^{-1/2}$ in the fourth equality. It follows that $\Inn(Q)\leq \sbl{\phi_x\ |\ x\in Q}$.
\end{proof}

A Lie ring $(Q,+,[\cdot,\cdot])$ is said to be \emph{uniquely $2$-divisible} if the abelian group $(Q,+)$ is uniquely $2$-divisible.

\begin{theorem}[Partial correspondence between Lie rings and automorphic loops]
\thmlabel{correspondence}
Suppose that $(Q,+,[\cdot,\cdot])$ is a uniquely $2$-divisible Lie ring satisfying \eqref{Eq:Wright} and \eqref{Eq:WrightAutomorphic}. Then $(Q,\dia)$ defined by
\begin{displaymath}
    x\diamond y = x+y-[x,y]
\end{displaymath}
is a $2$-divisible automorphic loop whose associated Bruck loop $(Q,\circ)$ is an abelian group; in fact, $(Q,\circ) = (Q,+)$.

Conversely, suppose that $(Q,\cdot)$ is a uniquely $2$-divisible automorphic loop whose associated Bruck loop $(Q,\circ)$ is an abelian group. Then $(Q,\circ,[\cdot,\cdot])$ defined by
\[
[x,y] = x\circ y\circ (xy)\inv\tag{$[\cdot,\cdot]$}
\]
is a uniquely $2$-divisible Lie ring satisfying \eqref{Eq:Wright} and \eqref{Eq:WrightAutomorphic}.

Furthermore, the two constructions are inverses of each other. Subrings (resp. ideals) of the Lie ring are subloops (resp. normal subloops) of the corresponding automorphic loop, and subloops (resp. normal subloops) closed under square roots are subrings (resp. ideals) of the corresponding Lie ring.
\end{theorem}
\begin{proof}
Suppose that $(Q,+,[\cdot,\cdot])$ is a uniquely $2$-divisible Lie ring satisfying \eqref{Eq:Wright} and \eqref{Eq:WrightAutomorphic}. By Proposition \ref{Pr:LieToAut}, $(Q,\diamond)$ is an automorphic loop. Note that $x\diamond x = 2x$, $x^{-1}=-x$ and $x^{1/2} = \frac{1}{2}x$. The multiplication in the Bruck loop $(Q,\circ)$ associated with $(Q,\diamond)$ therefore has the form $x\circ y = ((2y)(L_{-x}^\diamond)^{-1}R_x^\diamond)\frac{1}{2} = ((2y)R_x^\diamond (L_{-x}^\diamond)^{-1})\frac{1}{2}$, where the second equality follows by Proposition \prpref{commutes}. Showing $x\circ y = x+y$ is therefore equivalent to proving $(2y)\diamond x = (2y)R_x^\dia = (2x+2y)L_{-x}^\dia = (-x)\dia(2x+2y)$. But $(2y)\dia x = 2y+x-[2y,x]=(-x)+(2x+2y)-[-x,2x+2y] = (-x)\dia (2x+2y)$.

Conversely, suppose that $(Q,\cdot)$ is a uniquely $2$-divisible automorphic loop whose associated Bruck loop $(Q,\circ)$ is an abelian group. By \eqnref{op_corres}, we have $[x,y] = x\circ y\circ (xy)^{-1} = x\circ y \circ ((x)\phi_y\circ y)^{-1} = x\circ y \circ (x\inv)\phi_y \circ y\inv =
x\circ (x\inv)\phi_y$. Since $(x\inv)\phi_x = x\inv$, we have $[x,x] = 1$. Next,
\begin{displaymath}
[x,y]\circ [y,x] = x\circ (x\inv)\phi_y \circ y \circ (y\inv)\phi_x
    = x\circ (y\inv)\phi_x \circ y\circ (x\inv)\phi_y
    = (x)\phi_y\circ y\inv \circ y\circ (x\inv)\phi_y
    = 1,
\end{displaymath}
where we have used \eqnref{flip} in the third equality and $\phi_y\in\Aut(Q)\le\Aut(Q,\circ)$ in the last equality. For biadditivity, we compute
\begin{gather*}
[x\circ y,z] = x\circ y \circ [(x\circ y)\inv]\phi_z = x\circ (x\inv)\phi_z \circ y\circ (y\inv)\phi_z = [x,z]\circ [y,z],\\
[x,y\circ z] = [y\circ z,x]\inv = ([y,x]\circ [z,x])\inv = [y,x]\inv\circ [z,x]\inv = [x,y]\circ [x,z].
\end{gather*}

So far we have shown that $(Q,\circ,[\cdot,\cdot])$ is an alternating, biadditive (nonassociative) ring with underlying abelian group $(Q,\circ)$. In what follows the symbols $+$ and $-$ will refer to sums and differences of endomorphisms of $(Q,\circ)$. Rearranging the definition of $[\cdot,\cdot]$ and using the skew-symmetry, we have $xy = x\circ y\circ [x,y]\inv = y\circ (x)(\id_Q - \ad(y))$. Comparing this with \eqnref{op_corres}, we see that $\id_Q - \ad(x) = \phi_x$ and also $\id_Q + \ad(x) = \phi_{x\inv}$. In particular, property \eqref{Eq:Wright} holds.

Now using biadditivity, we have
\[
[(x)(\id_Q + \ad(z)),(y)(\id_Q + \ad(z))]
= [x,y]\circ [x,[y,z]]\circ [[x,z],y] \circ [[x,z],[y,z]]\,,
\]
and also
\[
[x,y](\id_Q + \ad(z)) = [x,y] \circ [[x,y],z].
\]
Since $\id_Q + \ad(x)=\phi_x \in \Aut(Q)\leq \Aut(Q,[\cdot,\cdot])$, the results of these two calculations are equal. Canceling common terms and rearranging using skew-symmetry, we obtain
\begin{equation}
\eqnlabel{lietmp}
[[x,y],z]\circ [[y,z],x]\circ [[z,x],y] = [[x,z],[y,z]]
\end{equation}
for all $x$, $y$, $z\in Q$. Since the left side of \eqnref{lietmp} is invariant under cyclic permutations of $x$, $y$, $z$, so is the right side, and so we have
\begin{equation}
\eqnlabel{lietmp2}
[[x,z],[y,z]] = [[y,x],[z,x]]
\end{equation}
for all $x$, $y$, $z\in Q$. Replace $x$ in this last identity with $x\circ u$ and use biadditivity to get
\[
[[x,z],[y,z]]\circ [[u,z],[y,z]] = [[y,x],[z,x]]\circ [[y,u],[z,x]] \circ [[y,x],[z,u]]\circ [[y,u],[z,u]].
\]
Canceling terms on both sides using \eqnref{lietmp2}, we obtain
$1 = [[y,u],[z,x]]\circ [[y,x],[z,u]]$ for all $x$, $y$, $z$, $u\in Q$. Taking $u = y$, we get
$1 = [[y,x],[z,y]]$, which is equivalent to \eqref{Eq:WrightAutomorphic}. It follows that the right side of \eqnref{lietmp} is equal to $1$, and so the Jacobi identity holds. Therefore, $(Q,\circ,[\cdot,\cdot])$ is a Lie ring satisfying \eqref{Eq:Wright} and \eqref{Eq:WrightAutomorphic}.

Let us now show that the two constructions are inverse to each other. Suppose that the constructions yield $(Q,\cdot)\mapsto (Q,\circ,[\cdot,\cdot])\mapsto (Q,\dia)$. Then $x\dia y = x\circ y\circ[x,y]^{-1}$, and since $(Q,\circ)$ is an abelian group and $x\circ y\circ (xy)^{-1}=[x,y]$, we conclude that $x\dia y = xy$. In the other direction, let $(Q,+,[\cdot,\cdot])\mapsto (Q,\diamond)\mapsto (Q,\circ,\lceil \cdot,\cdot\rceil)$, where $(Q,\circ)$ is the Bruck loop associated with $(Q,\dia)$. We have already shown that $(Q,\circ)=(Q,+)$. Then $\lceil x,y\rceil = x\circ y\circ (x\dia y)^{-1} = x+y-(x+y-[x,y])=[x,y]$.

Finally, we show the correspondence of substructures. Suppose that $(Q,\cdot)$ corresponds to $(Q,+,[\cdot,\cdot])$. Lemma \ref{Lm:LieWright} shows that the three loop operations of $(Q,\cdot)$ (in fact, of $(Q,\dia)$, but $(Q,\dia)=(Q,\cdot)$ here) can be expressed in terms of $+$ and $[\cdot,\cdot]$.

If $S$ is a subring of $(Q,+,[\cdot,\cdot])$, then since $S$ is closed under $+$ and $[\cdot,\cdot]$, it is a subloop of $(Q,\cdot)$. If $S$ is an ideal of $(Q,+,[\cdot,\cdot])$, then it is invariant under the mappings $\id_Q - \ad(x) = \phi_x$ for all $x\in Q$ and hence $S$ is invariant under $\Inn(Q,\cdot)$ by Lemma \lemref{phigenerate}.

If $S$ is a subloop of $(Q,\cdot)$ closed under square roots, then by Proposition \prpref{bruckloop} $S$ is a subgroup of $(Q,\circ)$. Therefore $S$ is a subring of $(Q,+,[\cdot,\cdot])$ by definition of the bracket.

Finally, if $S$ is a normal subloop of $(Q,\cdot)$, then $S$ is invariant under all mappings $\id_Q - \ad(x) = \phi_x$. But then $(S)\ad(x)\subseteq S$ for all $x\in Q$, and so $S$ is an ideal of $(Q,+,[\cdot,\cdot])$.
\end{proof}

We conclude with the observation that in the uniquely $2$-divisible case the condition \eqref{Eq:WrightAutomorphic} already implies that $(Q,+,[\cdot,\cdot])$ is solvable of derived length at most $2$.

\begin{lemma}\label{Lm:SolvabilityFollows}
Let $(Q,+,[\cdot,\cdot])$ be a uniquely $2$-divisible Lie ring. Then $Q$ satisfies \eqref{Eq:WrightAutomorphic} if and only if $[[Q,Q],[Q,Q]]=0$.
\end{lemma}
\begin{proof}
Clearly, if $[[Q,Q],[Q,Q]]=0$ then \eqref{Eq:WrightAutomorphic} follows. For the converse, suppose that $[[x,y],[z,y]]=0$ for all $x$, $y$, $z\in Q$. Replacing $y$ with $y+u$ and then using \eqref{Eq:WrightAutomorphic} itself to cancel terms, we obtain $[[x,y],[z,u]]+[[x,u],[z,y]] = 0$ or by skew-symmetry,
\begin{equation}
\eqnlabel{lietmp4}
[[x,y],[z,u]] = [[x,u],[y,z]]
\end{equation}
for all $x$, $y$, $z$, $u\in Q$. Now if we apply the identity \eqnref{lietmp4} to its own right hand side, we obtain $[[x,u],[y,z]] = [[x,z],[u,y]]$, which together with \eqnref{lietmp4} gives
\begin{equation}
\eqnlabel{lietmp5}
[[x,y],[z,u]] = [[x,z],[u,y]]
\end{equation}
for all $x$, $y$, $z$, $u\in Q$. On the other hand, \eqnref{lietmp4} is equivalent to
\begin{equation}
\eqnlabel{lietmp6}
[[x,y],[z,u]] = [[u,x],[z,y]]
\end{equation}
for all $x$, $y$, $z$, $u\in Q$, using skew-symmetry. If we apply \eqnref{lietmp6} to its own right hand side, we obtain $[[u,x],[z,y]] = [[y,u],[z,x]] = - [[x,z],[u,y]]$, using skew-symmetry in the last equality. This together with \eqnref{lietmp6} gives
\begin{equation}
\eqnlabel{lietmp7}
[[x,y],[z,u]] = -[[x,z],[u,y]]
\end{equation}
for all $x$, $y$, $z$, $u\in Q$. Comparing \eqnref{lietmp5} and \eqnref{lietmp7}, we have
\[
2[[x,y],[z,u]] = 0
\]
for all $x$, $y$, $z$, $u\in Q$. Since $(Q,+)$ is uniquely $2$-divisible, it follows that $[[x,y],[z,u]] = 0$ for all $x$, $y$, $z$, $u\in Q$.
\end{proof}

\section{Nilpotency and Solvability}
\seclabel{solvability}

In this section we prove the Odd Order Theorem for automorphic loops together with two other corollaries of Theorem \thmref{correspondence}. We start with automorphic loops of prime power order.

Let $p$ be a prime. By Corollary \corref{orderpisgroup}, an automorphic loop of order $p$ is isomorphic to $\mathbb Z_p$. The following result was first obtained by Cs\"org\H{o} \cite{Csorgo1}, using her signature method of connected transversals. We can now give a short proof based on Theorem \thmref{correspondence}. A proof that is both short and elementary remains elusive.

\begin{theorem}(Cs\"org\H{o})
\thmlabel{orderp2}
Let $p$ be a prime. Every automorphic loop of order $p^2$ is a group.
\end{theorem}

\begin{proof}
Let $Q$ be an automorphic loop of order $p^2$. Every loop of order $4$ is associative \cite{Pflugfelder}, so assume $p>2$. Bruck loops of order $p^2$ are groups \cite{Burn}. If $(Q,\circ)$ is cyclic, then so is $Q$, so assume $(Q,\circ)$ is elementary abelian. Theorem \thmref{correspondence} and Lemma \ref{Lm:SolvabilityFollows} give an associated solvable Lie ring $(Q,\circ,[\cdot,\cdot])$ of derived length at most $2$. Since $(Q,\circ)$ is an elementary abelian, $(Q,\circ,[\cdot,\cdot])$ is a $2$-dimensional Lie algebra over $GF(p)$. Over any field, there are, up to isomorphism, only two $2$-dimensional Lie algebras, one abelian and the other nonabelian \cite{Humph}. The nonabelian Lie algebra of dimension $2$ has a basis $\{x,y\}$ such that $[x,y] = y$. But then $y(\id + \ad(x)) = 0$ so that condition \eqref{Eq:Wright} is not satisfied. Thus $(Q,\circ,[\cdot,\cdot])$ must be an abelian Lie algebra, that is, $[x,y] = 0$ for all $x$, $y\in Q$. Then $xy = x\circ y$, that is, $Q$ is an abelian group.
\end{proof}

Commutative automorphic loops of order $p^k$ are centrally nilpotent when $p$ is an odd prime \cite{Csorgo1,JKV3}. Commutative automorphic loops of order $p^3$ were classified up to isomorphism in \cite{dBGV}. There are additional nonassociative noncommutative automorphic loops of order $p^3$, $p$ and odd prime. A class of such loops with trivial nucleus was obtained in \cite{JKV3}. In particular, when $p$ is an odd prime, automorphic loops of order $p^3$ need not be centrally nilpotent. Here we present the construction of \cite{JKV3} in a new way, using the corresponding Lie algebras:

\begin{example}
\exmlabel{orderp3}
Let $F$ be a field and fix $A\in GL(2,F)$. On $Q = F\times F^2$, define an operation $[\cdot,\cdot]$ by
\[
[(a,x),(b,y)] = (0,(ay-bx)A)
\]
for all $a$, $b\in F$, $x$, $y\in F^2$. (Note that we think of elements of $F^2$ as row vectors so that $A$ acts on the right.) Then it is straightforward to verify that $(Q,+,[\cdot,\cdot])$ is a Lie algebra satisfying \eqref{Eq:WrightAutomorphic}. Let $r_x = \mathrm{id}_Q - \mathrm{ad}(x)$, $\ell_x = \mathrm{id}_Q+\mathrm{ad}(x)$ be as before. In block matrix form, we have
\begin{align*}
(a,x)r_{(b,y)} &= (a,x+(bx-ay)A) = (a,x)\begin{pmatrix} 1 & -yA \\ 0 & I + bA \end{pmatrix} \\
\intertext{and}
(b,y)\ell_{(a,x)} &= (b,y+(bx-ay)A) = (b,y)\begin{pmatrix} 1 & xA \\ 0 & I - aA \end{pmatrix}\,,
\end{align*}
where $I$ is the $2\times 2$ identity matrix. Thus condition \eqref{Eq:Wright} will hold precisely when $\det(I + \mu A)\neq 0$ for all $\mu\in F$, that is, when the characteristic polynomial of $A$ has no roots in $F$. (See \cite{JKV3} for an interpretation of this in terms of anisotropic planes.) Assume this property now holds for $A$.

We show that the left/right nucleus of the corresponding loop $(Q,\dia)$ is trivial. By Proposition \ref{Pr:LieToAut}, $N_{\lambda}(Q,\dia)$ consists of all elements $(a,x)$ such that $[[(a,x),(b,y)],(c,z)] = (0,c(ay-bx)A) = (0,0)$ for all $b$, $c\in F$, $y$, $z\in F^2$. Thus $c(ay-bx)A = 0$. Since $A\in GL(2,F)$ and taking $c\neq 0$, we have $ay=bx$ for all $b\in F$, $y\in F^2$. Taking $b\neq 0$, $y=0$ implies $x = 0$, while taking $b=0$, $y\neq 0$ implies $a=0$. Thus $N_{\lambda}(Q,\dia)$ is trivial.

Consider the particular case $F = GF(p)$. If $p = 2$, then by Corollary \corref{glie}, we obtain a commutative automorphic loop $(Q,\dia)$ of exponent $2$ and order $8$. There is precisely one such loop with trivial center, first constructed in \cite{JKV2}. As discussed in \cite{JKV3}, if $p = 3$, then this construction gives two isomorphism classes of (noncommutative) automorphic loops depending on the choice of $A$, while if $p = 5$, there are three isomorphism classes. For $p > 5$, it is conjectured that there are precisely three isomorphism classes \cite[Conj. 6.5]{JKV3}.
\end{example}

Returning to general automorphic loops of order $p^3$, $p$ odd prime, there is much that is still unknown, but we can at least say that for $p = 3$, such automorphic loops are necessarily given by the construction of Proposition \ref{Pr:LieToAut}:

\begin{lemma}
\lemlabel{order27}
Let $Q$ be an automorphic loop of order $27$ and exponent $3$. Then $Q$ is constructed from a Lie algebra satisfying \eqref{Eq:Wright} and \eqref{Eq:WrightAutomorphic} by the construction $(\dia)$.
\end{lemma}
\begin{proof}
Every Bruck loop of exponent $3$ is a commutative Moufang loop \cite{Rob}. Moufang loops of order $3^n$ for $n\leq 3$ are associative. Thus the associated Bruck loop $(Q,\circ)$ is an elementary abelian $3$-group. By Theorem \thmref{correspondence}, we have an associated solvable Lie ring $(Q,\circ,[\cdot,\cdot])$ satisfying \eqref{Eq:Wright}, \eqref{Eq:WrightAutomorphic}. Since $(Q,\circ)$ is elementary abelian, $(Q,\circ,[\cdot,\cdot])$ is a Lie algebra over $GF(3)$. By Theorem \thmref{correspondence}, $(Q,\cdot)$ is equal to the loop $(Q,\dia)$ obtained from $(Q,\circ,[\cdot,\cdot])$ by $(\dia)$.
\end{proof}

Lemma \lemref{order27} cannot be easily extended to Bruck loops of order $p^3$ and exponent $p$ for $p > 3$ because there are nonassociative Bruck loops of such orders.

We now start working toward the Odd Order Theorem.

If $Q$ is a loop and $S\leq Q$, the \emph{relative multiplication group of} $S$, denoted by $\Mlt(Q;S)$, is the subgroup of $\Mlt(Q)$ generated by all $R_x$, $L_x$, $x\in S$. The \emph{relative inner mapping group of} $S$ is $\Inn(Q;S) = (\Mlt(Q;S))_1 = \Mlt(Q;S)\cap \Inn(Q)$.

\begin{lemma}
\lemlabel{odd-subloop}
Let $Q$ be a finite automorphic loop of odd order. A subloop $S$ of the associated Bruck loop $(Q,\circ)$ is a subloop of $Q$ if and only if $S h = S$ for every $h\in \Inn(Q;S)$.
\end{lemma}

\begin{proof}
The ``only if'' direction is trivial, so assume the hypothesis of the converse assertion. Fix $u$, $v \in S$. Since powers agree in $(Q,\circ)$ and $Q$, we have $u\inv, v\inv\in S$. Set $w = v^{1/2}$ and note that $v\in S$ as well. By Lemma \lemref{bruck-aut}, $\Aut(Q,\cdot)\le\Aut(Q,\circ)$. Thus $S$ also contains
\[
(u \circ w)^2 T_u = (uT_u \circ wT_u)^2 = (u \circ wT_u)^2
= (u\inv \ldiv [w T_u]^2) u = v T_u L_{u\inv}^{-1} R_u = v R_u^2 L_u\inv L_{u\inv}\inv ,
\]
using \eqnref{RLLR}. Since $L_{u\inv} L_u \in \Inn(Q)$, $S$ also contains $v R_u^2 = (u\circ w)^2 T_u L_{u\inv} L_u$. By induction, $v R_u^{2k}\in S$ for all integers $k$. Now let $2n+1$ be the order of $u$. Then $R_u^{2n+1} \in \Inn(Q)$, and so $S$ contains $v R_u^{2n+1} R_u^{-2n} = v u$, and also $v R_u^{2n+1} R_u^{2(-n-1)} = v \rdiv u$. Thus $S$ is closed under multiplication and right division. By the AAIP, $S$ is also closed under left division, and hence is a subloop.
\end{proof}

\begin{lemma}
\lemlabel{Bruck_characteristic}
Let $Q$ be a uniquely $2$-divisible automorphic loop, and let $(Q,\circ)$ be the associated Bruck loop. Then every characteristic subloop of $(Q,\circ)$ is a normal subloop of $Q$.
\end{lemma}

\begin{proof}
If $S$ is a characteristic subloop of $(Q,\circ)$, then by Lemma \lemref{bruck-aut}, $S$ is invariant under $\Inn(Q)$. By Lemma \lemref{odd-subloop}, $S$ is subloop of $Q$.
\end{proof}

\begin{theorem}[Odd Order Theorem]
\thmlabel{ft}
Every automorphic loop of odd order is solvable.
\end{theorem}

\begin{proof}
Let $Q$ be a minimal counterexample. If $1 < S\propnormal Q$, then by minimality, both $S$ and $Q/S$ are solvable automorphic loops of odd order. This contradicts the nonsolvability of $Q$. Therefore $Q$ is simple.

Let $(Q,\circ)$ be the associated Bruck loop and let $D$ denote the derived subloop of $(Q,\circ)$. By \cite[Thm. 14(b)]{Gl2}, $(Q,\circ)$ is solvable and so $D$ is a proper subloop. Since $D\Char(Q,\circ)$, it follows from Lemma \lemref{Bruck_characteristic} that $D\unlhd Q$. Since $Q$ is simple, $D = \{1\}$. Therefore $(Q,\circ)$ is an abelian group.

Now let $p$ be a prime divisor of $|Q|$ and let $M_p=\{x\in Q\ |\ x^p=1\}$. Then $M_p\Char(Q,\circ)$, and so by Lemma \lemref{Bruck_characteristic} again, $M_p\unlhd Q$. By Theorem \thmref{cauchy}, $M_p$ is nontrivial, and so since $Q$ is simple, $M_p = Q$. Thus $Q$ has exponent $p$, $(Q,\circ)$ has exponent $p$ by Proposition \prpref{bruckloop}, and $(Q,\circ)$ is an elementary abelian $p$-group.

By Theorem \thmref{correspondence}, $(Q,\circ,[\cdot,\cdot])$ defined by ($[\cdot,\cdot]$) is a Lie ring satisfying \eqref{Eq:Wright} and \eqref{Eq:WrightAutomorphic}. By Lemma \ref{Lm:SolvabilityFollows}, $(Q,\circ,[\cdot,\cdot])$ is solvable. Since $(Q,\circ)$ is an elementary abelian $p$-group, we may view $(Q,\circ,[\cdot,\cdot])$ as a finite dimensional Lie algebra over $GF(p)$. Since $Q$ is simple as a loop, Theorem \thmref{correspondence} also implies that $(Q,\circ,[\cdot,\cdot])$ is either simple as a Lie algebra or else is abelian. The former case contradicts the solvability of $(Q,\circ,[\cdot,\cdot])$, and so $(Q,\circ,[\cdot,\cdot])$ is abelian. But then $xy = x\circ y\circ [x,y] = x\circ y$, so that $Q$ is an abelian group, a contradiction with nonsolvability of $Q$.
\end{proof}

We remark that the proof of \cite[Thm. 14(b)]{Gl2} depends on the Feit-Thompson Odd Order Theorem for groups, and hence so does our proof of Theorem \thmref{ft}.

\section{Finite Simple Automorphic Loops}
\seclabel{simplicity}

The main open problem in the theory of automorphic loops is the existence or nonexistence of a nonassociative finite simple automorphic loop, \emph{cf.}, Problem \prbref{simple}. By Theorem \thmref{ft} and by the main results of \cite{GKN}, such a loop would be noncommutative and of even order, though not a $2$-loop.

Simple loops can be studied via primitive permutation groups thanks to this classic theorem of Albert \cite{albert}:

\begin{proposition}
\prplabel{albert}
A loop $Q$ is simple if and only if $\Mlt(Q)$ is primitive.
\end{proposition}

\begin{lemma}
\lemlabel{centralizer}
Let $Q$ be a simple nonassociative automorphic loop with inversion map $J$. If $J\ne\mathrm{id}_Q$ then $C_{\Mlt(Q)}(J) = \Inn(Q)$.
\end{lemma}

\begin{proof}
Since $Q$ is automorphic, $J$ commutes with every inner mapping. Therefore $\Inn(Q)\leq C_{\Mlt(Q)}(J)$. Since $\Mlt(Q)$ is primitive by Proposition \prpref{albert}, $\Inn(Q)$ is a maximal subgroup of $\Mlt(Q)$. Since $J \neq \id_Q$, there is $x\in Q$ such that $x\ne x^{-1}$, and so $xJL_x =1\ne x^{-2}=xL_xJ$. Hence $C_{Mlt(Q)}(J) \neq \Mlt{Q}$, and so the desired equality holds.
\end{proof}

Recall that the \emph{socle} $\Soc(G)$ of a group $G$ is the subgroup generated by the minimal normal subgroups of $G$.
By the O'Nan-Scott Theorem \cite[Thm. 4.1A]{DM}, the analysis of a finite primitive group $G$ divides into two cases depending on whether or not $\Soc(G)$ is regular.

\begin{proposition}
Let $Q$ be a finite simple nonassociative automorphic loop. Then the socle $\Soc(\Mlt(Q))$ is not regular.
\end{proposition}

\begin{proof}
Suppose $S = \Soc(\Mlt(Q))$ is regular. Recall that $J$ normalizes $\Mlt(Q)$ in $\Sym(Q)$ by Corollary \corref{AAIP}. Thus since $S$ is characteristic in $\Mlt(Q)$, $S$ is normalized by $J$. By Theorem \thmref{ft}, $|S| = |Q|$ is even. Thus $J$ fixes a nonidentity element $s \in S$. If $J \neq \mathrm{id}_Q$, then by Lemma \lemref{centralizer}, $s \in \Inn(Q)$. But then $(1)s = 1$, which contradicts the regularity of $S$. Therefore $J = \mathrm{id}_Q$ and so $Q$ has exponent $2$. By \cite[Thm. 6.2]{JKV1}, $Q$ has order a power of $2$ and then by \cite[Thm. 3]{GKN}, $Q$ is solvable, a contradiction.
\end{proof}

By the O'Nan-Scott Theorem, it follows that $\Mlt(Q)$ is of almost simple type, of diagonal type or of product type \cite{DM}.

Although the classification of finite simple automorphic loops remains open, results from group theory about characteristic subgroups hold analogously for characteristic subloops of automorphic loops with essentially the same proofs (\emph{cf}. the closing remarks of \cite{GKN}). Part (ii) of the following result is \cite[Thm. 2.2(ii)]{BP}. 

\begin{theorem}
\thmlabel{characteristic}
Let $Q$ be an automorphic loop.
\begin{enumerate}
\item[(i)] If $T \Char S\normal Q$, then $T \normal Q$.
\item[(ii)] Every characteristic subloop of $Q$ is normal.
\item[(iii)] If $Q$ is finite and characteristically simple, then $Q$ is a direct product of isomorphic simple loops.
\end{enumerate}
\end{theorem}

\begin{proof}
Every inner mapping leaves $S$ invariant, hence acts as an automorphism of $S$. Since $T$ is characteristic in $S$, $T\varphi = T$ for all $\varphi\in \Inn(Q)$. This establishes (i), and (ii) follows from (i) by taking $S = Q$. Now suppose $Q$ is finite and characteristically simple, and let $S=S_1$ be a minimal normal subloop. Consider the orbit $\{S_1,\ldots,S_m\}$ of $S$ under $\Aut(Q)$. Each $S_i$, being the image of a minimal normal subloop of $Q$ under an automorphism, is also a minimal normal subloop of $Q$. Since each $S_i\cap S_j$ is normal in $Q$, it follows from minimality that the subloops $S_i$ intersect pairwise trivially. Thus $S_1\cdots S_m$ is a direct product \cite{Bruck}. Since automorphisms map the direct factors of $S_1\cdots S_m$ to each other, the direct product is characteristic in $Q$. Thus $Q=S_1\cdots S_m$ because $Q$ is characteristically simple. Since $S$ is both a minimal normal subloop and a direct factor of $Q$, $S$ must be simple. This establishes (iii).
\end{proof}

\begin{corollary}
\corlabel{minimal_normal}
A minimal normal subloop of a finite automorphic loop is a direct product of isomorphic simple loops.
\end{corollary}

\begin{proof}
If $S$ is a minimal normal subloop of $Q$, then by Theorem \thmref{characteristic}(i), $S$ is characteristically simple, and we are done by Theorem \thmref{characteristic}(iii).
\end{proof}

\begin{proposition}
Let $Q$ be an automorphic loop.
\begin{enumerate}
\item[(i)] $Q^{(n)} \normal Q$ for each $n > 0$.
\item[(ii)] If $Q$ is solvable, then the \emph{derived series} $Q \backnormal Q' \backnormal Q'' \backnormal \cdots \backnormal Q^{(n)} = 1$ is a normal series, that is, $Q^{(k)} \normal Q$ for all $k > 0$.
\end{enumerate}
\end{proposition}

\begin{proof}
Each $Q^{(k)}$ is characteristic in $Q^{(k-1)}$ for all $k\geq 1$. By Theorem \thmref{characteristic}(i), each $Q^{(n)} \normal Q$. This proves (i), and (ii) follows from (i).
\end{proof}

\section{Split Middle Nuclear Extensions}
\seclabel{mid_nuc}

In this brief section we will examine automorphic loops which are split extensions by their middle nuclei. The following proposition shows that this notion can be defined in either of the usual group theoretic ways.

\begin{proposition}
\prplabel{split}
Let $Q$ be a loop with normal middle nucleus $N_{\mu}=N_{\mu}(Q)$. The following conditions are equivalent.
\begin{enumerate}
\item[(i)] The natural homomorphism $\eta: Q\to Q/N_{\mu}$ splits, that is, there is a homomorphism $\sigma: Q/N_{\mu} \to Q$ such that $\sigma\eta = \id_{Q/N_{\mu}}$.
\item[(ii)] There exists a subloop $S$ of $Q$ such that $Q = SN_{\mu}$ and $S\cap N_{\mu} = 1$.
\end{enumerate}
\end{proposition}

\begin{proof}
Assume (i) holds. Let $S = \sigma(Q/N_{\mu})$, which is a subloop of $Q$ since $\sigma$ is a homomorphism. Clearly $S\cap N_{\mu} = 1$. For $x\in Q$, let $s = (x)\eta\sigma = (x N_{\mu})\sigma$ and let $a = s\backslash x$. Then $(a)\eta = N_{\mu}$, that is, $a\in \ker(\eta) = N_{\mu}$. Therefore $Q = SN_{\mu}$ and (ii) holds.

Assume (ii) holds. Suppose $sa = tb$ for $s,t\in S$ and $a$, $b\in N_{\mu}$. Then $s = tb\cdot a\inv = t\cdot ba\inv$ since $a$, $b\in N_{\mu}$. Hence $t\ldiv s = ba\inv \in S\cap N_{\mu} = 1$, and so $t=s$ and $b=a$. Thus each $x\in Q$ has a unique factorization $x = sa$ for some $s\in S$, $a\in N_{\mu}$. In particular, the subloop $S$ is a complete set of left coset representatives of $N_{\mu}$. Therefore setting $(sN_{\mu})\sigma = s$ for each $s\in S$ yields a well-defined map $\sigma : Q/N_{\mu} \to Q$ with $(sN_{\mu})\eta\sigma = sN_{\mu}$. Finally $\sigma$ is a homomorphism by the definition of coset multiplication.
\end{proof}

We will say that an automorphic loop $Q$ is a \emph{split middle nuclear extension} (of $S$ by $N_{\mu}$) if either, and hence both, of the conditions of Proposition \prpref{split} hold. In automorphic loops, the multiplication in a split middle nuclear extension has a very specific form:

\begin{proposition}
\prplabel{mid_nuc}
Let $Q$ be an automorphic loop. For all $a$, $b\in N_{\mu}(Q)$ and all $x$, $y\in Q$,
\begin{equation}
\eqnlabel{mid_nuc_mult}
xa\cdot yb = xy\cdot ((a)T_y\cdot b)L_{y,x}\,.
\end{equation}
\end{proposition}

\begin{proof}
First we prove
\begin{equation}
\eqnlabel{mid_nuc_tmp}
a\cdot yb = ay\cdot b\,.
\end{equation}
Since $b\in N_{\mu}=N_{\mu}(Q)$, we have $T_b = R_b L_{b\inv}$. Thus we compute
\begin{displaymath}
    ay\cdot b = b\cdot (ay)T_b = b[(a)T_b\cdot (y)T_b] =  b[b\inv ab\cdot b\inv yb] = b[b\inv a\cdot yb] = (b\cdot b\inv a)\cdot yb = a\cdot yb,
\end{displaymath}
where we have used $T_b\in \Aut(Q)$ in the second equality, $b,b\inv\in N_{\mu}$ in the fourth equality and $b\inv a\in N_{\mu}$ in the fifth equality. This establishes \eqnref{mid_nuc_tmp}.

For \eqnref{mid_nuc_mult}, we compute
\[
xa\cdot yb = x(a\cdot yb) \byeqn{mid_nuc_tmp} x(ay\cdot b)
= x\cdot (y\cdot (a)T_y)b = x\cdot y((a)T_y\cdot b)
= xy\cdot ((a)T_y\cdot b)L_{y,x}\,,
\]
where we have used $(a)T_y\in N_{\mu}$ (since $N_{\mu} \normal Q$ by Proposition \prpref{nuclei}) in the fourth equality.
\end{proof}

\begin{corollary}
\corlabel{split_mult}
Let $Q$ be an automorphic loop which is a split middle nuclear extension $Q = SN_{\mu}$. Then for all $s$, $t\in S$, $a$, $b\in N_{\mu}$,
\begin{equation}
\eqnlabel{split_mult}
sa\cdot tb = st\cdot ((a)T_t\cdot b)L_{t,s}\,,
\end{equation}
where the right hand side is the unique factorization of the left side into an element $st\in S$ and an element $((a)T_t\cdot b)L_{t,s}\in N_{\mu}$.
\end{corollary}

Just as split extensions of groups (internal semidirect products) lead naturally to external semidirect products, so do split middle nuclear extensions of automorphic loops lead to an ``external'' construction of automorphic loops. The input data are a loop $S$, a group $N$, a mapping $\phi : S\to \Aut(N)$ satisfying $(1)\phi=1$ and a mapping $\alpha : S\times S\to \Aut(N)$ satisfying $(1,s)\alpha = (s,1)\alpha = 1$ for all $s\in S$. On $Q := S\times N$, we define operations by
\begin{align*}
(s,a)\cdot (t,b) &= (st,(a^{(t)\phi}b)^{(t,s)\alpha})\,, \\
(s,a)\ldiv (t,b) &= (s\ldiv t, (a\inv)^{(s\ldiv t)\phi} b^{((s\ldiv t,s)\alpha)\inv})\,, \\
(s,a)\rdiv (t,b) &= (s\rdiv t, (a^{((t,s\rdiv t)\alpha)\inv}b\inv)^{((t)\phi)\inv})\,.
\end{align*}
Then it is easy to show $(Q,\cdot,\ldiv,\rdiv)$ is a loop with neutral element $(1,1)$. To get an automorphic loop, it is necessary that $S$ be automorphic and there are various conditions which must be satisfied by $\phi$ and $\alpha$. It is straightforward to find these conditions by simply calculating inner mappings in $Q$ and assuming them to be automorphisms. However, the calculations and the conditions themselves are both lengthy and unenlightening in their full generality. Since we are only going to examine a special case in detail in the next section, we omit the general construction.

\section{Dihedral Automorphic Loops}
\seclabel{dihedral}

We begin with a construction of automorphic loops motivated by Corollary \corref{split_mult}.

\begin{proposition}
Let $(A,+)$ be an abelian group and fix $\alpha\in \Aut(A)$. Let $\Dih(A,\alpha)$ be defined on $\mathbb Z_2\times A$ by
\begin{equation}
\eqnlabel{gen_dihedral}
(i,u)\cdot (j,v) = (i + j, ((-1)^ju + v)\alpha^{ij}).
\end{equation}
Then $(\Dih(A,\alpha),\cdot)$ is an automorphic loop. If $\alpha\neq\id_A$, then $N_{\mu} = \{0\}\times A \cong A$.
\end{proposition}

\begin{proof}
Throughout the proof, the exponent of $\alpha$ in \eqnref{gen_dihedral} is calculated in $\mathbb Z_2$. Clearly $(0,0)$ is the neutral element. Setting
\begin{align*}
(i,u)\ldiv (j,v) &= (i+j, v\alpha^{-i(j+i)} - (-1)^{i+j}u ),\\
(i,u)\rdiv (j,v) &= (i+j, (-1)^j (u \alpha^{-(i+j)j} - v),
\end{align*}
it is straightforward to show that $\ldiv$ and $\rdiv$ satisfy the properties of divisions in a loop.

The generalized conjugation $T_{(i,u)}$ is given by
\[
(j,v)T_{(i,u)} = (j, (-1)^i v + (1-(-1)^j)u )\,,
\]
as can be readily checked. Note that this is independent of $\alpha$. We check that this is an automorphism. First,
\begin{align*}
[(j,v)\cdot(k,w)]T_{(i,u)} &= (j + k, ((-1)^k v + w)\alpha^{jk} )T_{(i,u)}\\
&= (j+k, (-1)^i((-1)^k v + w)\alpha^{jk} + (1-(-1)^{j+k})u)\\
&= (j+k, (-1)^{i+k}v\alpha^{jk} + (-1)^iw\alpha^{jk} + (1-(-1)^{j+k})u).
\end{align*}
On the other hand, $(j,v)T_{(i,u)}\cdot (k,w)T_{(i,u)} =$
\begin{align*}
&= (j, (-1)^i v + (1-(-1)^j)u )\cdot (k, (-1)^i w + (1-(-1)^k)u )\\
&= (j+k, [(-1)^k((-1)^i v + (1-(-1)^j)u) + (-1)^i w + (1-(-1)^k)u]\alpha^{jk} )\\
&= (j+k, (-1)^{i+k}v\alpha^{jk} + (-1)^iw\alpha^{jk} + h),
\end{align*}
where
\[
h = [(-1)^k(1-(-1)^j)+(1-(-1)^k)]u\alpha^{jk} =
(1-(-1)^{j+k})u\alpha^{jk}\,.
\]
Checking all four possibilities, we see that $(1 - (-1)^{j+k})u\alpha^{jk} = (1 - (-1)^{j+k})u$ for $j,k\in \mathbb{Z}_2$. Thus $T_{(i,u)}$ is an automorphism.

Next, we check that the left inner mappings $L_{(j,v),(i,u)}$ are automorphisms. A lengthy calculation gives
\[
(k,w)L_{(j,v),(i,u)} = (k,[(-1)^{j+k}u(\alpha^{-jk}-\id_A) + w]\alpha^{ij})\,.
\]
Note that this is independent of $v$. We have
\begin{align*}
    (k,w)&L_{(j,v),(i,u)}\cdot (\ell,x)L_{(j,v),(i,u)} \\
    &=
    (k,[(-1)^{j+k}u(\alpha^{-jk}-\id_A) + w]\alpha^{ij})\cdot (\ell,[(-1)^{j+\ell}u(\alpha^{-j\ell}-\id_A) + x]\alpha^{ij})\\
    &=(k+\ell,\{(-1)^{\ell}[(-1)^{j+k}u(\alpha^{-jk}-\id_A) + w] +
    (-1)^{j+\ell}u(\alpha^{-j\ell}-\id_A) + x \}\alpha^{ij}\alpha^{k\ell})\\
    &= (k+\ell, [(-1)^{\ell}w + x + q]\alpha^{ij}\alpha^{k\ell} ),
\end{align*}
where
\begin{align*}
q &= (-1)^{\ell}u[(-1)^{j+k}(\alpha^{-jk}-\id_A) + (-1)^j(\alpha^{-j\ell}-\id_A)]\\
&= (-1)^{j + k + \ell}u(\alpha^{-jk}-\id_A + (-1)^k(\alpha^{-j\ell}-\id_A ))\\
&= (-1)^{j + k + \ell}u(\alpha^{-j(k+\ell)} - \id_A).
\end{align*}
The last equality follows by checking all possible values of $j$, $k$, $\ell\in \mathbb{Z}_2$.
On the other hand, we compute
\begin{align*}
    [(k,w)&\cdot (\ell,x)]L_{(j,v),(i,u)} = \\
    &= (k+\ell,[(-1)^{\ell}w + x]\alpha^{k\ell})L_{(j,v),(i,u)}\\
    &= (k+\ell,\{(-1)^{j+k+\ell}u(\alpha^{-j(k+\ell)}-\id_A) + [(-1)^{\ell}w + x]\alpha^{k\ell}\}\alpha^{ij})\\
    &= (k+\ell,\{(-1)^{j+k+\ell}u(\alpha^{-j(k+\ell)}-\id_A)\alpha^{-k\ell} +(-1)^{\ell}w + x\}\alpha^{k\ell}\alpha^{ij}).
\end{align*}
Now observe that $u(\alpha^{-j(k+\ell)}-\id_A)\alpha^{-k\ell} = u(\alpha^{-j(k+\ell)}-\id_A)$ for all $j$, $k$, $\ell\in \mathbb{Z}_2$ just by checking all possibilities. Thus we see that $L_{(j,v),(i,u)}$ is an automorphism.

Applying Proposition \prpref{halfcheck}, we have shown that $\Dih(A,\alpha)$ is an automorphic loop. It remains to characterize the middle nucleus when $\alpha\neq\id_A$. We have that $(j,v)\in N_m$ if and only if $(k,w) = (k,w)L_{(j,v),(i,u)}$ for all $i$, $k\in \mathbb{Z}_2$, $u$, $w\in A$. Thus matching second components, we require
\begin{equation}
\eqnlabel{mid_nuc_chk}
[(-1)^{j+k}u(\alpha^{-jk} - \id_A) + w]\alpha^{ij} = w
\end{equation}
for all $i$, $k\in \mathbb{Z}_2$, $u$, $w\in A$. Taking $u = 0$, $i = 1$, we must have $w\alpha^j = w$ for all $w\in A$. Thus $\alpha^j = \id_A$. Since $\alpha\neq\id_A$, we must have $j=0$. On the other hand, since \eqnref{mid_nuc_chk} is independent of $v$, it is clear that $(0,v)\in N_{\mu}$. This completes the proof.
\end{proof}

We call the loops $\Dih(A,\alpha)$ \emph{generalized dihedral automorphic loops}. $\Dih(A,\id_A)$ is the usual generalized dihedral group determined by the abelian group $A$. If $A = \mathbb{Z}$, then $\Aut(A) = \mathbb{Z}^* = \{\pm 1\}$. In this case we write $D_{\infty}(c) = \Dih(\mathbb{Z},c)$ where $c = \pm 1$ and refer to these loops as \emph{infinite dihedral automorphic loops}. If $A = \mathbb{Z}_n$, then $\Aut(A) = \mathbb{Z}_n^*$, the group of integers in $\{1,\ldots,n-1\}$ coprime to $m$. We write $D_{2n}(c) = \Dih(\mathbb{Z}_n,c)$ where $c\in \mathbb{Z}_n^*$ and refer to these loops simply as \emph{dihedral automorphic loops}.

In $D_{\infty}(c)$ or $D_{2n}(c)$, the multiplication specializes as follows:
\begin{equation}\label{Eq:Dihedral}
(i,j)\cdot (k,\ell) = (i + k, c^{ik}((-1)^k j + \ell)),
\end{equation}
where $c\in \{\pm 1\}$ in the former case and $c\in \mathbb{Z}_n^*$ in the latter case.

We now show that different values of the parameter $c$ give nonisomorphic dihedral automorphic loops, and we calculate their automorphism groups.

\begin{lemma}\label{Lm:DihedralGens}
Let $Q=D_{2n}(c)$. Then
\begin{enumerate}
\item[(i)] $(0,1)^m = (0,m)$ for every $m\in \mathbb Z$, and $\mathbb Z_n\cong 0\times \mathbb Z_n\le Q$,
\item[(ii)] $|(1,x)|=2$ for every $x\in\mathbb Z_n$,
\item[(iii)] $(1,0)\cdot(0,y) = (1,y)$ for every $y\in\mathbb Z_n$, and $Q = \langle (0,1),\,(1,0)\rangle$.
\end{enumerate}
\end{lemma}
\begin{proof}
(i) Since automorphic loops are power-associative, the power $(0,1)^m$ is well-defined for every $m\in\mathbb Z$. The claim holds for $m=0$, since $(0,0)$ is the neutral element of $Q$. Suppose the claim holds for some $m\ge 0$. Then $(0,1)^{m+1} = (0,1)^m\cdot(0,1) = (0,m)\cdot(0,1) = (0,m+1)$. Since $(0,-m)\cdot(0,m)=(0,0)$, it follows that $(0,1)^{-m} = (0,-m)$. The rest is clear.

(ii) For any $x\in\mathbb Z_n$ we have $(1,x)\cdot(1,x) = (0,c(-x+x)) = (0,0)$.

(iii) The formula $(1,0)\cdot(0,y) = (1,y)$ follows immediately from \eqref{Eq:Dihedral}. Then $Q=\langle (0,1),(1,0)\rangle$ follows from (i).
\end{proof}

By Lemma \ref{Lm:DihedralGens}, a loop homomorphism $f:D_{2n}(c)\to Q$ is determined by its values $(0,1)f$, $(1,0)f$. If $n>2$ then $0\times\mathbb Z_n$ is the unique subloop of $D_{2n}(c)$ isomorphic to $\mathbb Z_n$, by Lemma \ref{Lm:DihedralGens}(iii). Hence, if $f:D_{2n}(c)\to D_{2n}(d)$ is an isomorphism, it follows that $(0,1)f=(0,\alpha)$ for some $\alpha\in \mathbb Z_n^*$, and $(1,0)f=(1,\beta)$ for some $\beta\in\mathbb Z_n$. Using Lemma \ref{Lm:DihedralGens} again, we then have
\begin{align*}
    (0,x)f &= ((0,1)^x)f = ((0,1)f)^x = (0,\alpha)^x = (0,x\alpha),\\
    (1,x)f &= ((1,0)\cdot(0,x))f = (1,0)f\cdot(0,x)f = (1,\beta)\cdot(0,x\alpha) = (1,\beta+x\alpha)
\end{align*}
for every $x\in\mathbb Z_n$.

Given any $\alpha\in\mathbb Z_n^*$, $\beta\in\mathbb Z_n$, let us denote the mapping $f:D_{2n}(c)\to D_{2n}(d)$ satisfying $(0,x)f = (0,x\alpha)$, $(1,x)f = (1,\beta+x\alpha)$ for all $x\in\mathbb Z_n$ by $f_{\alpha,\beta}$. (Note that the definition of $f_{\alpha,\beta}$ does not require knowledge of $c$, $d$, so we will consider $f_{\alpha,\beta}$ to be a mapping from $D_{2n}(c)$ to $D_{2n}(d)$ for any $c$, $d\in\mathbb Z_n^*$.)

\begin{lemma}\label{Lm:NearHomomorphism}
Let $c$, $d\in\mathbb Z_n^*$, $\alpha\in Z_n^*$ and $\beta\in \mathbb Z_n$. Then $f=f_{\alpha,\beta}:D_{2n}(c)\to D_{2n}(d)$ is a bijection that satisfies $((0,x)\cdot(0,y))f = (0,x)f\cdot(0,y)f$, $((0,x)\cdot(1,y))f = (0,x)f\cdot(1,y)f$ and $((1,x)\cdot(0,y))f = (1,x)f\cdot(0,y)f$ for every $x$, $y\in\mathbb Z_n$. Moreover, $f$ is an isomorphism if and only if $c=d$.
\end{lemma}
\begin{proof}
Since $\alpha\in\mathbb Z_n^*$, it is clear from the definition of $f=f_{\alpha,\beta}$ that it is a bijection $D_{2n}(c)\to D_{2n}(d)$. For $x$, $y\in\mathbb Z_n$ we have $((0,x)\cdot(0,y))f = (0,x+y)f = (0,(x+y)\alpha) = (0,x\alpha)\cdot(0,y\alpha) = (0,x)f\cdot(0,y)f$, $((0,x)\cdot(1,y))f = (1,-x+y)f = (1,\beta + (-x+y)\alpha) = (0,x\alpha)\cdot(1,\beta+y\alpha) = (0,x)f\cdot(1,y)f$, and $((1,x)\cdot(0,y))f = (1,x+y)f = (1,\beta + (x+y)\alpha) = (1,\beta+x\alpha)\cdot(0,y\alpha) = (1,x)f\cdot(0,y)f$. Finally, we have $((1,x)\cdot(1,y))f = (0,c(-x+y))f = (0,c(-x+y)\alpha)$, while $(1,x)f\cdot (1,y)f = (1,\beta + x\alpha)\cdot(1,\beta+y\alpha) = (0,d(-(\beta+x\alpha)+\beta+y\alpha)) = (0,d(-x+y)\alpha)$, so $f$ is an isomorphism if and only if $c=d$.
\end{proof}

\begin{corollary}\label{Cr:DihedralNonisomorphic}
For an integer $n\ge 2$, the loops $D_{2n}(c)$, $c\in\mathbb Z_n^*$ are pairwise nonisomorphic.
\end{corollary}

\begin{proposition}
\label{Pr:AutomorphismGroup}
Let $c\in\mathbb Z_n^*$ and $Q=D_{2n}(c)$. Then $\mathrm{Aut}(Q)$ is isomorphic to the holomorph $\mathrm{Aut}(\mathbb Z_n)\rtimes \mathbb Z_n = \mathbb Z_n^*\rtimes \mathbb Z_n$ with multiplication $(\alpha,\beta)(\gamma,\delta) = (\alpha\gamma,\beta + \alpha\delta)$.
\end{proposition}
\begin{proof}
By the discussion preceding Lemma \ref{Lm:NearHomomorphism}, every automorphism of $Q$ is of the form $f_{\alpha,\beta}$ for some $\alpha\in\mathbb Z_n^*$, $\beta\in\mathbb Z_n$. By Lemma \ref{Lm:NearHomomorphism}, every such mapping $f_{\alpha,\beta}$ is an automorphism of $Q$. Now, if $\gamma\in\mathbb Z_n^*$, $\delta\in\mathbb Z_n$ and $x\in\mathbb Z_n$, we have $(0,x)f_{\gamma,\delta}f_{\alpha,\beta} = (0,x\gamma)f_{\alpha,\beta} = (0,x\gamma\alpha) = (0,x\alpha\gamma) = (0,x)f_{\alpha\gamma,\beta+\alpha\delta}$ and $(1,x)f_{\gamma,\delta}f_{\alpha,\beta} = (1,\delta+x\gamma)f_{\alpha,\beta} = (1,\beta+(\delta +x\gamma)\alpha) = (1,\beta+\alpha\delta + x\alpha\gamma) = (1,x)f_{\alpha\gamma,\beta+\alpha\delta}$.
\end{proof}

Results analogous to \ref{Lm:DihedralGens}--\ref{Pr:AutomorphismGroup} hold for the infinite dihedral automorphic loops $D_\infty(c)$, with every occurrence of $\mathbb Z_n$ replaced with $\mathbb Z$, and $2n$ replaced with $\infty$.

Commutative automorphic loops with middle nuclei of index $2$ were studied in detail in \cite{JKV2}. In the next result we examine the noncommutative case under the assumption that the middle nucleus is cyclic.

\begin{proposition}
\prplabel{index2}
Let $Q$ be a noncommutative automorphic loop with cyclic middle nucleus $N_{\mu}(Q) = \langle b\rangle$, and suppose that $Q$ is a split middle nuclear extension $Q = \langle a\rangle \langle b\rangle$ where $a^2 = 1$. If $Q$ is infinite, then $Q \cong D_{\infty}(c)$ for some $c\in \{\pm 1\}$. If $Q$ is finite, then $Q \cong D_{2n}(c)$ for some $n\in \mathbb{N}$ and some $c\in \mathbb{Z}_n^*$.
\end{proposition}

\begin{proof}
Since $T_a^2 = \id_Q$ by Lemma \lemref{RxyLxy} (see \eqnref{Tinv}), we must have $(b)T_a = b$ or $(b)T_a = b\inv$ by the normality of $\langle b\rangle$ in $Q$. If the former situation holds, then $(a^i b^j)T_a = a^i b^j$ for all $i=0$, $1$ and all $j$ since $T_a$ is an automorphism. Therefore $T_a$ fixes every point of $Q$ and hence $a\in C(Q)$. It follows that $(a^i b^j)T_b = a^i b^j$ for all $i=0$, $1$ and all $j$ since $T_b$ is an automorphism. Thus $b\in C(Q)$. Therefore $C(Q) = Q$, that is, $Q$ is commutative, a contradiction. It follows that $(b)T_a = b\inv$.

We have that $T_1 = L_{1,1} = L_{1,a} = L_{a,1} =  \id_Q$, and so referring to \eqnref{split_mult}, we see that the multiplication in $Q$ is entirely determined by the automorphism $L_{a,a} \upharpoonright \langle b\rangle$. If $\langle b\rangle \cong \mathbb{Z}$, then $\Aut(\langle b\rangle) \cong \mathbb{Z}^* = \{\pm 1\}$ and there are two possible values for $L_{a,a} \upharpoonright \langle b\rangle$ determined by $(b)L_{a,a} = b^c$ where $c = \pm 1$. If $\langle b\rangle \cong \mathbb{Z}_n$, then $\Aut(\langle b\rangle) \cong \mathbb{Z}_n^*$, and the possible values for $L_{a,a} \upharpoonright \langle b\rangle$ are given by $(b)L_{a,a} = b^{c}$ where $c\in \mathbb{Z}_n^*$. In either case, we thus have $(b)L_{a^i,a^k} = b^{c^{ik}}$ for $i, k = 0,1$.

Fixing $c\in \mathbb{Z}^*$ or $\mathbb{Z}_n^*$, it follows from the preceding discussion that \eqnref{split_mult} specializes to the present setting as follows:
\begin{equation}
\eqnlabel{index2}
a^i b^j \cdot a^k a^{\ell} = a^{i+k} b^{c^{ij}((-1)^k j + \ell)}\,,
\end{equation}
for all $i$, $k\in \mathbb{Z}_2$, $j$, $\ell\in \mathbb{Z}$ or $\mathbb{Z}_n$. Finally, for $m = \infty$ or $2n$, define $\psi : D_m(c)\to Q$ by $(i,j)\psi = a^i b^j$. It is straightforward to check that $\psi$ is an isomorphism using \eqref{Eq:Dihedral} and \eqnref{index2}.
\end{proof}

As an application, we have the following classification results.

\begin{theorem}
\thmlabel{order2modd}
Let $Q$ be a finite automorphic loop with a cyclic subgroup of odd order $n$ and of index $2$. Then either $Q$ is a cyclic group or $Q\cong D_{2n}(c)$ for some $c\in \mathbb{Z}_n^*$.
\end{theorem}

\begin{proof}
Let $\langle b\rangle$ be a cyclic subgroup of order $n$. This subloop is normal in $Q$ since it has index $2$. By Corollary \corref{even}, $Q$ also has an element $a$ of order $2$. By \eqnref{Tinv}, $(b)T_a = b$ or $(b)T_a = b\inv$, and by the same argument as in the proof of Proposition \prpref{index2}, we see that the former case leads to $Q$ being commutative. If $Q$ is commutative, then by \cite[Thm. 5.1]{JKV1}, $Q$ is isomorphic to the direct product $\mathbb{Z}_2 \times \mathbb{Z}_n\cong \mathbb{Z}_{2n}$. Thus we assume from now on that $Q$ is noncommutative, and so $(b)T_a = b\inv$.

It remains to show that $\langle b\rangle$ is the middle nucleus of $Q$. Since $Q$ is the disjoint union of $\langle b\rangle$ and $a\langle b\rangle$, every element of $Q$ has a unique representation in the form $a^i b^j$, $i = 0,1$, $0\leq j < n$. Thus to show $\langle b\rangle \subseteq N_\mu(Q)$, we must show $(a^i b^j\cdot b^k)\cdot a^{\ell} b^r =
a^i b^j\cdot (b^k\cdot a^{\ell} b^r)$ for all $0\leq i$, $\ell \leq 1$, $0\leq j$, $k$, $r < n$.

Our first step is to prove
\begin{equation}
\eqnlabel{basecase}
bab = a\,.
\end{equation}
Set $c = b^{(n+1)/2}$ so that $c^2 = b$.
We use \eqnref{PRnorm} to get $(x\inv)P_{xy} = (x\inv)R_{x\inv}\inv P_y L_x = xy^2$ for all $x$, $y\in Q$. Take $x = a\rdiv c$ and $y = c$ in this to get $(a\rdiv c)b = (a\rdiv c)c^2 = [(a/c)\inv ]P_a = [(a/c)\inv ]T_a$ because $P_a = T_a$ since $a^2 = 1$. We record this as $(a\rdiv c)b = [(a\rdiv c)\inv ]T_a$, and use this identity twice in the following:
\begin{align*}
b\cdot(a\rdiv c)b &= b\cdot [(a/c)\inv ]T_a = (b\inv)T_a \cdot [(a/c)\inv ]T_a = [b\inv (a\rdiv c)\inv]T_a\\
 &= [((a\rdiv c)b)\inv]T_a = [((a\rdiv c)b)T_a]\inv = (((a\rdiv c)\inv)T_a T_a)^{-1} = a\rdiv c,
\end{align*}
where we also used $T_a\in \Aut(Q)$ in the third and fifth equalities, and AAIP in the fourth. Hence $aR_c = aR_c R_b L_b = aR_b L_b R_c$ by Proposition \prpref{commutes}. Canceling, we obtain \eqnref{basecase}.

Recall that we work under the assumption $(b)T_a= b^{-1}$. By \eqnref{basecase}, we have $a = bab = (a\cdot (b)T_a)b = ab\inv\cdot b$. Thus the automorphism $L_a R_b L_a\inv R_b\inv$ fixes $b\inv$ and hence fixes each $b^k$, that is, $ab^k\cdot b = ab^{k+1}$ for $0\leq k < n$. Then the automorphism $L_{b^k} L_a L_{ab^k}\inv$ fixes $b$ and hence fixes each $b^r$, that is, $ab^k\cdot b^r = ab^{k+r}$ for $0\leq k, r < n$. Since $ab^k \cdot a = a\cdot b^ka$ (by Proposition \prpref{commutes}), $L_{b^k} L_a L_{ab^k}\inv$ also fixes $a$ and hence fixes each $a^{\ell}b^r$, that is, $ab^k\cdot a^{\ell}b^r = a(b^k\cdot a^{\ell}b^r)$ for $0\le \ell \le 1$, $0 \leq k,r < n$.

On the other hand, by \eqnref{basecase} again, we have $a = bab = b((b)T_a\inv\cdot a) = b((b)T_a\cdot a) = b\cdot b\inv a$. Dualizing the arguments of the preceding paragraph, we get $b^j (b^k\cdot a^{\ell}b^r) = b^{j+k}\cdot a^{\ell}b^r$ for $\ell = 0,1$, $0\leq j,k,r < n$. Combining this with the preceding paragraph, we see that $R_{b^k} R_{a^{\ell}b^r} R_{b^k\cdot a^{\ell}b^r}\inv$ fixes both $a$ and each $b^j$. It follows that $(a^i b^j\cdot b^k)\cdot a^{\ell} b^r = a^i b^j\cdot (b^k\cdot a^{\ell} b^r)$ for $0\leq i,\ell\leq 1$, $0\leq j,k,r < n$, as desired.

We have shown that $\langle b\rangle \subseteq N_\mu(Q)$. If $ab^i\in N_\mu(Q)$ for any $i$, then $a\in N_\mu(Q)$ since $N_\mu(Q)$ is a subloop. But then $Q = N_\mu(Q)$, a contradiction. Therefore $\langle b\rangle = N_\mu(Q)$. By Proposition \prpref{index2}, we have the desired result.
\end{proof}

Recently, P. Cs\"{o}rg\H{o} was able to establish the following result by group-theoretic means:

\begin{theorem}[Elementwise Lagrange Theorem \cite{Csorgo2}]
\thmlabel{elementwise_Lagrange}
Let $Q$ be a finite automorphic loop and let $a\in Q$. Then the order of $a$ divides the order of $Q$.
\end{theorem}

\begin{corollary} [Automorphic loops of order $2p$]
\corlabel{order2p}
Let $Q$ be an automorphic loop of order $2p$ where $p$ is an odd prime. Then $Q\cong D_{2p}(c)$ for some integer $1 \le c < p$, or $Q\cong \mathbb Z_{2p}$. Thus there are precisely $p$ automorphic loops of order $2p$, including the cyclic group $\mathbb Z_{2p}$ and the dihedral group $D_{2p}$.
\end{corollary}

\begin{proof}
By Corollary \corref{even}, $Q$ has an element $a$ of order $2$. If every element of $Q$ had order $2$, then by \cite[Thm. 8]{GKN}, $Q$ itself would have order a power of $2$, a contradiction. By Theorem \thmref{elementwise_Lagrange}, every element of $Q$ has order dividing $2p$. Thus $Q$ must have an element $b$ of order $p$. Since $\langle a\rangle \cap \langle b\rangle = 1$, the desired isomorphism now follows from Theorem \thmref{order2modd}.

The $p-1$ dihedral automorphic loops $D_{2p}(c)$, $c\in\mathbb Z_p^*$ are pairwise nonisomorphic by Corollary \ref{Cr:DihedralNonisomorphic}.
\end{proof}

\section{Open Problems}
\seclabel{problems}

The main open problem in the theory of automorphic loops is the following:

\begin{problem}
\prblabel{simple}
Does there exist a (finite) simple, nonassociative automorphic loop?
\end{problem}

Also open are the Lagrange, Cauchy, Sylow and Hall theorems.

\begin{problem}
\prblabel{lagrange}
Let $Q$ be a finite automorphic loop and let $S\leq Q$. Does $|S|$ divide $|Q|$?
\end{problem}

\begin{problem}
\prblabel{sylow}
Let $Q$ be a finite automorphic loop.
\begin{enumerate}
\item[(i)] For each prime $p$ dividing $|Q|$, does $Q$ have an element of order $p$?
\item[(i)] For each prime $p$ dividing $|Q|$, does $Q$ have a Sylow $p$-subloop?
\item[(i)] If $Q$ is solvable and if $\pi$ is a set of primes, does $Q$ have a Hall $\pi$-subloop?
\end{enumerate}
\end{problem}

\begin{acknowledgment}
Our investigations were aided by the automated deduction tool \textsc{Prover9} and the finite model builder \textsc{Mace4}, both developed by McCune \cite{McCune}, and the LOOPS package \cite{LOOPS} for GAP \cite{GAP}. We thank Ian Wanless for the idea behind the proof of Lemma \lemref{odd-2div}(ii). We thank G\'{a}bor Nagy for remarking that standard group theory facts about characteristic subgroups should hold for characteristic subloops of automorphic loops.
\end{acknowledgment}


\begin{thebibliography}{99}
\bibitem{albert} A. A. Albert, Quasigroups I,
    \textit{Trans. Amer. Math. Soc.} \textbf{54} (1943), 507--519.

\bibitem{Bruck} R. H. Bruck,
    \textit{A Survey of Binary Systems},
    Springer, 1971.

\bibitem{BP} R. H. Bruck and L. J. Paige,
    Loops whose inner mappings are automorphisms,
    \textit{Ann. of Math.} (2) \textbf{63} (1956) 308--323.

\bibitem{Burn} R. P. Burn,
    Finite Bol loops,
    \textit{Math. Proc. Cambridge Philos. Soc.} \textbf{84} (1978), 377--385.

\bibitem{Csorgo1} P. Cs\"{o}rg\H{o},
    Multiplication groups of commutative automorphic p-loops of odd order are p-groups,
    \textit{J. Algebra} \textbf{350} (2012), 77--83.

\bibitem{Csorgo2} P. Cs\"{o}rg\H{o},
    All finite automorphic loops have the elementvise Lagrange property,
    preprint.

\bibitem{dBGV} D. A. S. de Barros, A. Grishkov and P. Vojt\v{e}chovsk\'y,
    Commutative automorphic loops of order $p^3$,
    to appear in \textit{Journal of Algebra and its Applications}.

\bibitem{DM} J. Dixon and B. Mortimer
    \textit{Permutation Groups},
    Graduate Texts in Math. \textbf{163}, Springer, 1996.

\bibitem{Drapal} A. Dr\'{a}pal,
    A-loops close to code loops are groups,
    \textit{Comment. Math. Univ. Carolin.} \textbf{41} (2000), 245--249.

\bibitem{FKP} T. Foguel, M. Kinyon and J. D. Phillips,
    On twisted subgroups and Bol loops of odd order,
    \textit{Rocky Mountain J. Math} \textbf{36} (2006), 183--212.

\bibitem{GAP} The GAP Group,
    GAP -- Groups, Algorithms, and Programming,
    Version 4.4.10; 2007. \url{http://www.gap-system.org}

\bibitem{Gl1} G. Glauberman,
    On loops of odd order I.
    \textit{J. Algebra} \textbf{1} (1964), 374--396.

\bibitem{Gl2} G. Glauberman,
    On loops of odd order II.
    \textit{J. Algebra} \textbf{8} (1968), 393--414.

\bibitem{GKN} A. Grishkov, M. Kinyon and G. Nagy,
    Solvability of commutative automorphic loops,
    submitted, \url{arXiv:1111.7138}.

\bibitem{Humph} J. E. Humphreys,
    \textit{Introduction to Lie algebras and representation theory},
    Graduate Texts in Mathematics \textbf{9}, Springer-Verlag, New York, 1978.

\bibitem{JKV1} P. Jedli\v{c}ka, M. Kinyon and P. Vojt\v{e}chovsk\'{y},
    The structure of commutative automorphic loops,
    \textit{Trans. Amer. Math. Soc.} \textbf{363} (2011), no. 1, 365--384.

\bibitem{JKV2} P. Jedli\v{c}ka, M. Kinyon and P. Vojt\v{e}chovsk\'{y},
    Constructions of commutative automorphic loops,
    \textit{Comm. Algebra} \textbf{38} (2010), no. 9, 3243--3267.

\bibitem{JKV3} P. Jedli\v{c}ka, M. Kinyon and P. Vojt\v{e}chovsk\'{y},
    Nilpotency in automorphic loops of prime power order,
    \textit{J. Algebra} \textbf{350} (2012), 64--76.

\bibitem{JKNV} K. W. Johnson, M. K. Kinyon, G. P. Nagy and P. Vojt\v{e}chovsk\'{y},
    Searching for small simple automorphic loops,
    \textit{LMS J. Comput. Math.} \textbf{14} (2011), 200--213.

\bibitem{KKP} M. K. Kinyon, K. Kunen, and J. D. Phillips,
    Every diassociative A-loop is Moufang,
    \textit{Proc. Amer. Math. Soc.} \textbf{130} (2002), 619--624.

\bibitem{McCune} W. W. McCune,
    \textit{Prover9 and Mace4}, version 2009-11A.
    \url{http://www.cs.unm.edu/~mccune/prover9/}

\bibitem{LOOPS} G. P. Nagy and P. Vojt\v{e}chovsk\'y,
    \emph{LOOPS: Computing with quasigroups and loops in GAP},
    version 2.0.0, computational package for GAP;
    \url{http://www.math.du.edu/loops}

\bibitem{Osborn} J. M. Osborn,
    A theorem on A-loops,
    \textit{Proc. Amer. Math. Soc.} \textbf{9} (1958), 347--349.

\bibitem{Pflugfelder} H. O. Pflugfelder,
    \textit{Quasigroups and Loops: Introduction},
    Sigma Series in Pure Math. \textbf{7}, Heldermann, 1990.

\bibitem{Rob} D. A Robinson,
    Bol quasigroups,
    \textit{Publ. Math. Debrecen} \textit{19} (1972), 151--153.


\bibitem{Wright1} C. R. B. Wright,
    Nilpotency conditions for finite loops,
    \textit{Illinois J. Math.} \textbf{9} (1965), 399--409.

\bibitem{Wright2} C. R. B. Wright,
    On the multiplication group of a loop,
    \textit{Illinois J. Math.} \textbf{13} (1969), 660--673.

\end{thebibliography}
\end{document}